\documentclass[a4paper, reqno, 10pt]{amsart}

\usepackage{stackengine}
\usepackage{verbatim}
\usepackage{systeme}
\usepackage{enumitem}
\usepackage[dvips]{color}
\usepackage[usenames,dvipsnames]{xcolor}
\usepackage{mathtools}
\usepackage{graphicx}

\usepackage{algorithm}
\usepackage{algpseudocode}

\usepackage{amssymb}

\usepackage{mathtools}
\usepackage[active]{srcltx}
\usepackage[top=1in, bottom=1in, left=1in, right=1in]{geometry}
	\usepackage[colorlinks,pdfpagelabels,pdfstartview = FitH,bookmarksopen = true,bookmarksnumbered = true,linkcolor = NavyBlue,plainpages = false,hypertexnames = false,citecolor = OliveGreen] {hyperref}

\allowdisplaybreaks

\date{\today}

\title[HJB]{On the stability of Lipschitz continuous control problems and its application to reinforcement learning}

\author[Cho]{Namkyeong Cho}
\address{Center for Mathematical Machine Learning and its Applications, POSTECH, 77 Cheongam-Ro, Nam-Gu, Pohang, Gyeongbuk, Korea}
 \email{namkyeong.cho@gmail.com}

\author[Kim]{Yeoneung Kim}
\address{Department of Applied Artificial Intelligence,
SeoulTech, 232 Gongneung-ro, Nowon-gu, Seoul, Korea}
 \email{kimyeoneung@gmail.com}

\newtheorem{theorem}{Theorem}[section]

\newtheorem{proposition}[theorem]{Proposition}
\newtheorem{lemma}[theorem]{Lemma}
\newtheorem{cor}[theorem]{Corollary}

\theoremstyle{definition}
\newtheorem{definition}[theorem]{Definition}
\newtheorem{remark}[theorem]{Remark}

\newtheorem{assumption}[theorem]{Assumption}
\numberwithin{equation}{section}

\def\eqn#1$$#2$${\begin{equation}\label#1#2\end{equation}}
\def\charfn_#1{{\raise1.2pt\hbox{$\chi_{\kern-1pt\lower3pt\hbox{{$\scriptstyle#1$}}}$}}}

\makeatletter
\newcommand{\pushright}[1]{\ifmeasuring@#1\else\omit\hfill$\displaystyle#1$\fi\ignorespaces}
\newcommand{\pushleft}[1]{\ifmeasuring@#1\else\omit$\displaystyle#1$\hfill\fi\ignorespaces}
\makeatother

\def\mean#1{\mathchoice%
          {\mathop{\kern 0.2em\vrule width 0.6em height 0.69678ex depth -0.58065ex
                  \kern -0.8em \intop}\nolimits_{\kern -0.4em#1}}%
          {\mathop{\kern 0.1em\vrule width 0.5em height 0.69678ex depth -0.60387ex
                  \kern -0.6em \intop}\nolimits_{#1}}%
          {\mathop{\kern 0.1em\vrule width 0.5em height 0.69678ex
              depth -0.60387ex
                  \kern -0.6em \intop}\nolimits_{#1}}%
          {\mathop{\kern 0.1em\vrule width 0.5em height 0.69678ex depth -0.60387ex
                  \kern -0.6em \intop}\nolimits_{#1}}}

\def\vintslides_#1{\mathchoice%
          {\mathop{\kern 0.1em\vrule width 0.5em height 0.697ex depth -0.581ex
                  \kern -0.6em \intop}\nolimits_{\kern -0.4em#1}}%
          {\mathop{\kern 0.1em\vrule width 0.3em height 0.697ex depth -0.604ex
                  \kern -0.4em \intop}\nolimits_{#1}}%
          {\mathop{\kern 0.1em\vrule width 0.3em height 0.697ex depth -0.604ex
                  \kern -0.4em \intop}\nolimits_{#1}}%
          {\mathop{\kern 0.1em\vrule width 0.3em height 0.697ex depth -0.604ex
                  \kern -0.4em \intop}\nolimits_{#1}}}

\newcommand{\aveint}[2]{\mathchoice%
          {\mathop{\kern 0.2em\vrule width 0.6em height 0.69678ex depth -0.58065ex
                  \kern -0.8em \intop}\nolimits_{\kern -0.45em#1}^{#2}}%
          {\mathop{\kern 0.1em\vrule width 0.5em height 0.69678ex depth -0.60387ex
                  \kern -0.6em \intop}\nolimits_{#1}^{#2}}%
          {\mathop{\kern 0.1em\vrule width 0.5em height 0.69678ex depth -0.60387ex
                  \kern -0.6em \intop}\nolimits_{#1}^{#2}}%
          {\mathop{\kern 0.1em\vrule width 0.5em height 0.69678ex depth -0.60387ex
                  \kern -0.6em \intop}\nolimits_{#1}^{#2}}}

\makeatletter
\def\subsubsection{\@startsection{subsubsection}{3}%
  \z@{.5\linespacing\@plus.7\linespacing}{.1\linespacing}%
  {\normalfont\itshape}}
\makeatother

\newtoks\by
\newtoks\paper
\newtoks\book	
\newtoks\jour
\newtoks\yr
\newtoks\pages
\newtoks\vol
\newtoks\publ

\def\ota{{\hbox{\bf ???}}}
\def\cLear{\by=\ota\paper=\ota\book=\ota\jour=\ota\yr=\ota
\pages=\ota\vol=\ota\publ=\ota}
\def\endpaper{\the\by, \textit{\the\paper},
{\the\jour} \textbf{\the\vol} (\the\yr), \the\pages.\cLear}
\def\endbook{\the\by, \textit{\the\book},
\the\publ, \the\yr.\cLear}
\def\endpap{\the\by, \textit{\the\paper}, \the\jour.\cLear}
\def\endproc{\the\by, \textit{\the\paper}, \the\book, \the\publ,
\the\yr, \the\pages.\cLear}

\begin{document}
\maketitle
\begin{abstract} 

We address the crucial yet underexplored stability properties of the Hamilton--Jacobi--Bellman (HJB) equation in model-free reinforcement learning contexts, specifically for Lipschitz continuous optimal control problems. We bridge the gap between Lipschitz continuous optimal control problems and classical optimal control problems in the viscosity solutions framework, offering new insights into the stability of the value function of Lipschitz continuous optimal control problems. By introducing structural assumptions on the dynamics and reward functions, we further study the rate of convergence of value functions. Moreover, we introduce a generalized framework for Lipschitz continuous control problems that incorporates the original problem and leverage it to propose a new HJB-based reinforcement learning algorithm. The stability properties and performance of the proposed method are tested with well-known benchmark examples in comparison with existing approaches.

\end{abstract}

\section{Introduction}
Reinforcement learning (RL) is known to be an effective approach for sequential decision-making problems or optimal control problems, particularly for discrete-in-time settings~\cite{watkins1992q,mnih2015human}. One of the well-known data-driven approaches to tackle such problems is $Q$-learning~\cite{watkins1992q}, which is built upon Bellman's dynamic programming principle~\cite{bellman1966dynamic}. However, the classical $Q$-learning and its variants such as the Deep-$Q$-learning algorithm are often limited to discrete-in-time problems. For the extension of RL algorithm to the continuous-in-time problem, various methods have been proposed such as~\cite{jia2023q,kim2021hamilton,yildiz2021continuous}. In~\cite{jia2023q}, authors provide a rigorous justification of $Q$-function for continuous stochastic optimal control problem under the presence of entropy-regularizer in the reward function, which is motivated by~\cite{wang2020reinforcement}. On the other hand, ~\cite{kim2021hamilton} proposed a model-free reinforcement learning algorithm called Hamilton--Jabobi deep $Q$-learning (HJDQN) to tackle continuous-in-time deterministic optimal control problems based on the viscosity solution framework. Their idea is to introduce a value function that contains both initial state and action restricting the action to be Lipschitz continuous while the value function only includes the state variable and is defined to be an optimal reward or cost corresponding to the initial state given in standard optimal control theory. This way, one can admit the viscosity solution to the corresponding Hamilton--Jacobi equation as a Q-function which consists of a pair of an initial state and an action, denoted by $Q^L(x,a)$ for some $L>0$ given. However, such an extension is meaningful only if we restrict the class of admissible controls to be $L$-Lipschitz continuous, and hence, the value function depends on $L$. It is also studied in the paper that the choice of hyperparameter $L$ is essential for better performance in practical applications while the optimal choice of $L$ is unknown. Nevertheless, their approach is still a stepping stone to understanding continuous time $Q$-learning through the lens of the theory of viscosity solution.

In this paper, we rigorously analyze the stability property of $Q^L(x,a)$ on the Lipschitz constraint as well as the convergence of $Q^L(x,a)$ as $L$ grows to infinity. We also demonstrate the rate of convergence under some structural assumptions on the dynamics and reward function. Furthermore, we introduce a slightly general landscape for admitting $Q$-function as a solution to a Hamilton--Jacobi--Bellman equation by considering a general norm and extending the existing HJDQN algorithm accordingly. Our idea is to consider different metric when defining the Lipschitz continuity of action, which allows flexibility for choosing actions.

\subsection*{Organization of paper}
The paper is organized as follows. In Section~\ref{sec:pre}, the problem setup and motivation are presented. We then provide some stability and regularity results on the viscosity solution to the Hamilton--Jacobi--Bellman equations via the theory of viscosity solutions in Section~\ref{sec:quant}. In the following section, a rate of convergence is discussed. Finally, we introduce a different class of compact control to generalize the Lipschitz continuous optimal control problem and present empirical results in Section~\ref{sec:gen} and Section~\ref{sec:num} respectively. We then conclude with Section~\ref{sec:con}.
\subsection*{Notations}
We need a set of notations used throughout the paper. Let $n,k \in\mathbb {N}$, $x=(x_1,...,x_n)\in\mathbb{R}^n$ and $p \geq 1$.
\begin{itemize}
\item  We denote $\|x\|_p :=  (\sum_{i=1}^n |x_i|^p)^{1/p})$. \item $D_x f(x):= (\frac{\partial f}{\partial x_1},...,\frac{\partial f}{\partial x_n})$ and $\Delta f := \sum_{i=1}^n \frac{\partial^2 f}{\partial x_i^2}$.

\item Define $k$-dimensional ball as
\[
B_{a,k}(y):=\{x \in \mathbb{R}^k: \|x-y\|_2 \leq a\},
\]
for $y\in\mathbb{R}^k$.
\item For $k$-continuosly differentiable function $f:\mathbb{R}^n \mapsto \mathbb{R}$,  
\[
\|f\|_{C^k(\mathbb{R}^n)} := \sum_{\alpha_1+...+\alpha_n\leq k}  \sup_{\mathbb{R}^n} \|D^\alpha f(x)\|_2.
\]

\item We denote
\begin{equation}\label{eq:mollifier}
\eta^\varepsilon(x):= \frac{1}{\varepsilon} \eta (\frac{x}{\varepsilon}), 
\end{equation}
where
\[
\eta(x):=\begin{cases}
e^{-1/(1-\|x\|_2^2)}\quad\text{for}\quad \|x\|_2\leq 1,\\
0\quad\text{else}.
\end{cases}
\]

\end{itemize}

\subsection*{Acknowledgement}
This work was supported by the National Research Foundation of Korea (NRF) grant funded by the Korea government(MSIT) (RS-2023-00219980) and the second author is partially supported by National Research Foundation of Korea (NRF) grant funded by the Korea government(MSIT) (RS-2023-00211503). Authors are thankful to Hung Vinh Tran at the University of Wisconsin-Madison for sharing ideas and discussions.
\section{Setup and Preliminaries}\label{sec:pre}
Let us introduce the collection of control 
\begin{equation}\label{eq:controlset}
    \mathcal{A}:=\{  a(\cdot) : [0, \infty) \to \mathbb{R}^m , a(\cdot) \textit{ is measurable}\}.
\end{equation}
Given $a(\cdot) \in \mathcal{A}$, we consider a dynamics evolving under
\[
\begin{cases}
x'(s)=f(x(s),a(s)),\\
x(0)=x,
\end{cases}
\]
where $x(t)$ is the system state and $f(x,a):\mathbb{R}^n \mapsto \mathbb{R}^m \rightarrow \mathbb{R}^n$ 
is smooth. Hence, it is known that $x(t)$ is absolutely continuous. Then, the standard infinite-horizon discounted optimal control problem is formulated as
\[
Q(x) = \sup_{a\in \mathcal{A}}\left\{ \int_0^\infty e^{-\gamma s} r(x(s),a(s)) \mathrm{d}s : x(0)=x \in \mathbb{R}^n \right\},
\]
where $\gamma>0$ denotes the discount factor, and $r(x,a):\mathbb{R}^n \times \mathbb{R}^m \rightarrow \mathbb{R}$ is a smooth reward function.

Given $L>0$, we introduce the class of the Lipschitz controls denoted by
\begin{equation}\label{eq:lipcontrol}
\mathcal{A}^{L}:=\{  a(\cdot) \in \mathcal{A}: \|a(s_1))- a(s_2)\|_2 \leq L|s_1-s_2| \quad \text{for all}\quad s_1,s_2\in [0, \infty) \},
\end{equation}
and define the value function as
\begin{equation}\label{eq:Q_L_dynamic}
Q^L(x,a) = \sup_{a\in \mathcal{A}^{L}}\left\{ \int_0^\infty e^{-\gamma s } r(x(s),a(s)) \mathrm{d}s : x(0)=x , a(0) =a\right\}.
\end{equation}

\begin{assumption}\label{ass:naive}
Throughout the paper, we assume that $f$ and $r$ are Lipschitz continuous, that is, there exists $C>0$ such that
\[
\|f\|_{L^\infty(\mathbb{R}^n\times \mathbb{R}^m)}+\|r\|_{L^\infty(\mathbb{R}^n\times \mathbb{R}^m)} 
+\|f\|_{\text{Lip}(\mathbb{R}^n\times \mathbb{R}^m)}
+
\|r\|_{\text{Lip}(\mathbb{R}^n\times \mathbb{R}^m)}
\leq C.
\]
\end{assumption}
Under the assumption above, it is known that $Q(x)$ and $Q^L(x,a)$ solve
\[
\gamma Q - \sup_{a\in\mathbb{R}^m} (D_x Q \cdot f(x,a)+r(x,a))=0,
\]
and
\begin{equation}\label{eq:ql}
\gamma Q^L - D_x Q^L \cdot f(x,a) - L\|D_a Q^L\|_2 -r(x,a)=0,
\end{equation}
respectively in viscosity sense~\cite{tran2021hamilton,bardi1997optimal,kim2021hamilton}. 

We end this section by reminding of the regularity property of $Q$. A further investigation on $Q^L$ will be presented in the following section. In the classical infinite-horizon optimal control problem, it is often assumed that the control takes values in a compact set as opposed to our setting. Nevertheless, we still have the same regularity property.

\begin{proposition}\label{prop:lip_reg_q}
Under Assumption~\ref{ass:naive}, the unique viscosity solution $Q$ to   
\[
\gamma Q - \sup_{a\in\mathbb{R}^m} (D_x Q \cdot f(x,a)+r(x,a))=0
\]
is Lipschitz continuous if $\gamma>  \|f\|_{\text{Lip}(\mathbb{R}^n\times \mathbb{R}^m)}$.
\end{proposition}
\begin{proof}
See~\cite{tran2021hamilton}.
\end{proof}
In the following section, it is shown that $Q^L$ is also Lipschitz continuous. In addition to that we study the stability of $Q^L$ in $L$ as well as the convergence as $L$ grows to infinity.



\section{Some quantitative properties of $Q^L$}\label{sec:quant}
In this section, we explore some quantitative behaviors of $Q^L$ by employing standard methods~\cite{tran2021hamilton} used for analyzing the regularity properties of viscosity solutions to Hamilton-Jacobi-Bellman (HJB) equations. We first show that $Q^L$ is uniformly Lipschitz in the state and action variable and derive the estimate for the rate of change in $L$. 

\subsection{Lipschitz Regularity}
Let us begin by providing a result on Lipschitz continuity of $Q^L$ in state and action variables. The following lemma suggests that the Lipschitz constant is independent of $L$. To show this, as demonstrated in~\cite{kim2021hamilton}, we introduce a new state variable $z:=(x,a)$ and control variable $b(\cdot):=\dot a(\cdot)$ such that $\|b(\cdot)\|_2 \leq L$. 

\begin{lemma}\label{lem:Lip_regualrity}
Let $L>0$ be given and suppose that $Q^L$ is the solution of \eqref{eq:ql}, and $f$ and $r$ satisfy Assumption \ref{ass:naive}. 
In addition, we assume that
\begin{equation*}
\gamma > \|f\|_{\text{Lip}(\mathbb{R}^n \times \mathbb{R}^m)}.
\end{equation*}
Then, $Q^L(x,a)$ is Lipschitz continuous in $x$ and $a$. Furthermore, we have 
the estimate
\begin{equation}\label{eq:lipschitz_est}
|Q^L|+\|D_xQ^{L}\|_2+L\|D_aQ^{L}\|_2 \leq  C
\end{equation}
for some $C>0$. 
\end{lemma}
\begin{proof}
Let $\mathcal{B}^L:=\{b(\cdot):[0,\infty)\rightarrow \mathbb{R}^m,b(\cdot)\text{ is measurable and } \|b(\cdot)\|_2 \leq L\}$. Given $z:=(x,a) \in \mathbb{R}^{n+m}$ and $b(\cdot)\in\mathcal{B}^L$, let us define
\[
V_{b(\cdot)}^L(x,a) := \int_{0}^\infty e^{-\gamma s} r(x(s),a(s)) \mathrm{d}s,
\]
subject to
\[
\begin{cases}
x'(s)=f(x(s),a(s))\quad\text{for}\quad s>0,\\
a'(s)= b(s)\quad\text{for}\quad s>0,\\
x(0)=x,\\
a(0)=a.
\end{cases}
\]
Fixing $\tilde z=(\tilde x,\tilde a) \in\mathbb{R}^{n+m}$ and $b(\cdot) \in \mathcal{B}^L$ as well as the trajectory $\tilde z(s):=(\tilde x(s),\tilde a(s))$ satisfying
\[
\begin{cases}
\tilde x' (s) =f(\tilde x(s),\tilde a(s))\quad\text{for}\quad s>0,\\
\tilde a' (s) = b(s)\quad\text{for}\quad s>0,\\
\tilde x(0)=\tilde x,\\
\tilde a(0)=\tilde a,
\end{cases}
\]
and letting $z(s):=(x(s),a(s))$, we have that
\begin{equation*}
\|z'(s)-\tilde z'(s)\|_2 \leq \|f\|_{\text{Lip}(\mathbb{R}^n \times \mathbb{R}^m)}\|z(s)-\tilde z(s) \|_2.
\end{equation*}
For simplicity, let $\gamma_0:=\|f\|_{\text{Lip}(\mathbb{R}^n \times \mathbb{R}^m)}$. Invoking Gronwall's inequality, we have that 
\begin{equation*}
\|z(s)-\tilde z(s)\|_2 \leq e^{\gamma_0 s} \|z(0)-\tilde z(0)\|_2 = e^{\gamma_0 s}(\|x-\tilde x\|_2+ \|a-\tilde a\|_2).
\end{equation*}
Now observing
\begin{equation*}
\begin{split}
|V^L_{b(\cdot)}(x,a))-V^L_{b(\cdot)}(\tilde x,\tilde a)| &= \bigg|\int_0^\infty e^{-\gamma s} r(x(s),a(s))\mathrm{d}s - \int_0^\infty e^{-\gamma s}r(\tilde x(s),\tilde a(s)) \mathrm{d} s \bigg |\\
& \leq \|r\|_{\text{Lip}(\mathbb{R}^n \times \mathbb{R}^m)} \int_0^\infty e^{-\gamma s} (\|x(s)-\tilde x(s)\|_2+\|a(s)-\tilde a(s)\|_2)\mathrm{d} s\\
&\leq \|r\|_{\text{Lip}(\mathbb{R}^n \times \mathbb{R}^m)} (\|x-\tilde x\|+\|a-\tilde a\|_2)\int_0^\infty e^{(-\gamma +\gamma_0)s}\mathrm{d} s\\
&\leq  C (\|x-\tilde x\|_2+\|a-\tilde a\|_2),
\end{split}
\end{equation*}
for some $C>0$ since $\gamma > \gamma_0$, we deduce that 
\[
|Q^L(x,a) - Q^L(\tilde x,\tilde a)| \leq C (\|x-\tilde x\|_2+\|a-\tilde a\|_2)
\]
by taking supremum over $b\in\mathcal{B}^L$. 

The inequality \eqref{eq:lipschitz_est} follows from the equation \eqref{eq:ql} and we finish the proof.
\end{proof}

\subsection{Error Bounds of $Q^{L+\ell}-Q^{L}$ for $L,\ell \geq 0$.}
Using the Lipschitz estimate above, one can estimate the rate of change of $Q^L$ with respect to $L$ based on the standard doubling variable argument~\cite{tran2021hamilton,kim2020state,evans2022partial,crandall1983viscosity}. The result we propose implies that the value function $Q^L$ does increase abruptly as $L$ increases. Numerical justification for this property is also presented in Section~\ref{sec:num}.
\begin{lemma}\label{lem:stepwise_converge}
Let Assumption \ref{ass:naive} be enforced and $\gamma > \|f\|_{\text{Lip}(\mathbb{R}^n \times \mathbb{R}^m)}$. Given $L>0$ and $\ell>0$, let $Q^L(x,a)$ and $Q^{L+\ell}(x,a)$ be solutions to \eqref{eq:ql} repsectively where the latter is associated with $L+\ell$ instead of $L$. Then, there exists a constant $C$ satisfying 
\begin{equation}\label{eq:Q_L_bound}
0 \leq Q^{L+\ell}(x,a) - Q^L(x,a) \leq \frac{C\ell}{L+\ell},
\end{equation}
and hence, 
\[
\limsup_{\ell\rightarrow 0} \frac{Q^{L+\ell} (x,a)-Q^L(x,a)}{\ell} \leq \frac{C}{L}.
\]
\end{lemma}
\begin{proof}
Since $Q^L$ is defined as \eqref{eq:Q_L_dynamic} and $\mathcal{A}^{L}$, the set of Lipschitz continuous controls, increases in $L$, it  follows that
\[
Q^L(x,a) \leq Q^{L+\ell}(x,a) \quad \text{for every $(x,a) \in \mathbb{R}^{n} \times \mathbb{R}^{m}$.}
\]
Therefore, the left side of inequality \eqref{eq:Q_L_bound} holds.
 
To show the right side of inequality \eqref{eq:Q_L_bound}, we use the doubling variable technique~\cite{tran2021hamilton,achdou2013introduction,crandall1983viscosity,kim2020state}.
Let $\mu(x) \in C^1(\mathbb{R}^n)$ $\tilde{\mu}(a)\in C^1(\mathbb{R}^{m})$ satisfy
\begin{equation}\label{eq:mu_property}
\begin{aligned}
	\|D_x\mu\|_{2} &\leq 1\quad\text{and} \quad\| D_a\tilde{\mu}\|_{2} &\leq 1.
\end{aligned}
\end{equation}
For $\delta, \varepsilon>0$ and $(x,y,a,b)\in \mathbb{R}^{2n} \times \mathbb{R}^{2m}$, we define an auxiliary function as follows:
\begin{align*}
\Psi(x,y,a,b)	& := Q^{L+ \ell}(x,a) - Q^{L}(y,b) - \frac{\|x-y\|_{2}^{2} +\|a  - b\|_{2}^{2} }{2\varepsilon} 
-\delta \left(\mu(x)+\mu(y) + \tilde{\mu}(a) +  \tilde{\mu}(b)
\right).
\end{align*}
We choose a point \((\bar{x}, \bar{y},\bar{a} ,\bar{b}) \in \mathbb{R}^{2n} \times \mathbb{R}^{2m} \) satisfying
\[
\Psi(\bar{x}, \bar{y}, \bar{a},  \bar{b}) = \sup_{(x,y,a,b)\in \mathbb{R}^{2n}\times \mathbb{R}^{2m}}\{\Psi(x,y,a, b) : (x,y,a, b)\in \mathbb{R}^{2n} \times \mathbb{R}^{2m} \}.
\]
From the inequality $ \Psi(\bar{y},\bar{y},\bar{a} ,\bar{b})
\leq 
\Psi(\bar{x},\bar{y},\bar{a} ,\bar{b})$, we see that 
\begin{align*}
	&
  Q^{L+\ell}(\bar{y},\bar{a}) 
  -\delta \mu(\bar{y})  \leq 
 Q^{L+\ell}(\bar{x},\bar{a})  - \frac{\|  \bar{x}-\bar{y} \|_{2}^{2}  }{2\varepsilon}  -\delta \mu(\bar{x})  
.
\end{align*}
Recalling Lemma~\ref{lem:Lip_regualrity} and \eqref{eq:mu_property}, we establish that
\begin{align*}
\frac{ \|\bar{x}-\bar{y} \|_{2}^{2} }{2\varepsilon} &\leq Q^{L}(\bar{x},\bar{a}) -Q^{L}(\bar{y}, \bar{a}) -\delta (\mu(\bar{y}) - \mu(\bar{x}) )\\
&\leq C(1+\delta) \|\bar{x}-\bar{y}\|_{2}
\end{align*}
for some constant $C>0$. 
Rearranging inequality above gives
\begin{equation}\label{eq:x_y}
    \|\bar{x}-\bar{y}\|_{2} \leq C(1+\delta)\varepsilon.
\end{equation}
Similarly, we start with the inequality
$\Psi(\bar{x},\bar{y}, \bar{b}, \bar{b}) \leq \Psi(\bar{x}, \bar{y}, \bar{a}, \bar{b})$. Invoking Lemma~\ref{lem:Lip_regualrity} and \eqref{eq:mu_property} once again,
\begin{align*}\frac{
\|\bar{a} - \bar{b}\|_2^2
}{2\varepsilon} 
	 &\leq
	 Q^{L+\ell}(\bar{x},\bar{a} ) - Q^{L+\ell}(\bar{x}, \bar{b})+\delta(\tilde{\mu}(\bar{b}) -\tilde{\mu}(\bar{a})  ) \\
	 &\leq \frac{C\|\bar{a} - \bar{b}\|_{2} }{L+\ell} +\delta \|\bar{a}-\bar{b}\|_{2}.
	 \end{align*}
Dividing both side by $\|\bar{a} - \bar{b}\|_2$, we derive that
\begin{equation}\label{eq:a_b}
\|\bar{a} - \bar{b}\|_{2}  \leq  \left(\frac{C}{L+\ell} +2\delta \right)\varepsilon.
	 \end{equation}
Since the map
\[
(x,a)\mapsto
 Q^{L+\ell}(x, a) - \frac{\|x-\bar{y}\|^2_{2} + \|a-\bar{b}\|^{2}_{2} } {2\varepsilon} -\delta(\mu(x) +\tilde{\mu}(a))
\]
obtains a local maximum value at $(\bar{x}, \bar{a})$, we have that
\begin{equation}\label{eq:L_l_upper_bound}
\gamma Q^{L+\ell}(\bar{x},\bar{a}) - f(\bar{x},\bar{a})\cdot\left( 
\frac{\bar{x}-\bar{y}}{\varepsilon}
+\delta D_x\mu(\bar{x})
\right)
-(L+\ell)
\left\|\frac{\bar{a}-\bar{b}}{\varepsilon} +\delta D_a\tilde{\mu}(\bar{a})\right\|_{2} -r(\bar{x},\bar{a})\leq 0,
\end{equation}
by the subsolution test.

Similarly, the map
$$
(y,b) \mapsto Q^{L}(y,b)-
\left(-\frac{\|y-\bar{x}\|_{2}^2 + \|b-\bar{a}\|_{2}^2}{2\varepsilon}
-\delta (\mu(y)+\tilde{\mu}(b)) 
\right)
$$
obtains a local minimum value at $(\bar{y},\bar{b})$. 
From the definition of viscosity supersolution, we see that
\begin{equation}\label{eq:L_lower_bound}
\gamma Q^{L}(\bar{y},\bar{b}) +
f(\bar{y},\bar{b})\cdot\left( 
\frac{\bar{y}-\bar{x}}{\varepsilon}
-\delta D_x\mu(\bar{y})
\right)
-L
\left\|\frac{\bar{a}-\bar{b}}{\varepsilon} +\delta D_a\tilde{\mu}(\bar{b})\right\|_{2} -r(\bar{y},\bar{b
})\geq 0.
\end{equation}
Subtracting \eqref{eq:L_lower_bound} from 
 \eqref{eq:L_l_upper_bound}, we have
\begin{equation}\label{eq:L+l_L}
    \begin{aligned}
        \gamma & \left(
        Q^{L+\ell}(\bar{x}, \bar{a} )
        -
        Q^{L}(\bar{y}, \bar{b} )
        \right)\\
        &\leq  (f(\bar{x},\bar{a}) - f(\bar{y}, \bar{b}) )\cdot\left(\frac{\bar{x}-\bar{y}}{\varepsilon}\right) + (r(\bar{x},\bar{a})- r(\bar{y},\bar{b}))  + \ell \left\|\frac{\bar{a}-\bar{b}}{\varepsilon}\right\|_2 \\
        &\quad 
         +\delta\left(\|f(\bar{x},\bar{a})\|_{2} \|D_x\mu(\bar{x})\|_{2}+ \|f(\bar{y},\bar{b})\|_{2}\|D_x\mu(\bar{y})\|_{2} + (L+\ell)\|D_a\tilde{\mu}(\bar{a})\|_{2} +L
\|D_a\tilde{\mu}(\bar{b })\|_{2}\right).
    \end{aligned}
\end{equation}
Combining \eqref{eq:mu_property}, \eqref{eq:x_y} and \eqref{eq:a_b}, there exists a constant $C$ satisfying
\begin{align*}
   \gamma &\left(
        Q^{L+\ell}(\bar{x}, \bar{a} )
        -
        Q^{L}(\bar{y}, \bar{b} )
        \right) \\
        &\leq C\left( \|\bar{x}-\bar{y}\|_{2} + \|\bar{a}-\bar{b}\|_{2}\right) \frac{\|\bar{x}-\bar{y}\|_{2}}{\varepsilon}+ C\left( \|\bar{x}-\bar{y}\|_{2} +\|\bar{a}-\bar{b}\|_{2} \right) +\ell\left\|\frac{\bar{a}-\bar{b}}{\varepsilon}\right\|_{2}  + \delta \left(C+ 2L+\ell\right) \\
        & \leq C\left(\frac{\|\bar{x}-\bar{y}\|_2}{\varepsilon}+1\right)(\|\bar{x}-\bar{y}\|_2 + \|\bar{a}-\bar{b}\|_2 ) + \ell \left\|\frac{\bar{a}-\bar{b}}{\varepsilon}\right\|_{2}+\delta (C+2L+\ell)\\
        &\leq  C\varepsilon +\frac{C\ell}{L+\ell} +\delta (C+2L+\ell).
\end{align*}
Finally, invoking the inequality
\[
\Psi(x,x,a,a) \leq \Psi(\bar{x}, \bar{y}, \bar{a}, \bar{b}) \quad \forall (x,a)\in \mathbb{R}^{n} \times \mathbb{R}^{m},
\]
we have that
\[
\gamma(
Q^{L+\ell}(x,a) -Q^{L}(x,a
)) -2 \delta(\mu(x) +\tilde{\mu}(a)) \leq  C\varepsilon +\frac{C\ell}{L+\ell} +\delta (C+2L+\ell).
\]
We complete the proof by taking $\delta, \varepsilon \to 0$.
\end{proof}

\section{Convergence of $Q^L(x,a)$ to $Q(x)$}
\subsection{General convergence result}\label{sub:convergence}
We provide the general convergence result based on the stability properties of viscosity solutions. Let us recall the following definition of upper or lower half-relaxed limit~\cite{mitake2016dynamical,achdou2013introduction}. It is crucial to note that the limit functions are subsolution and supersolution to the limiting equation.

\begin{definition}
For a family of locally bounded functions on $\mathbb{R}^n$ denoted by $\{u_\alpha\}_{\alpha\in\mathbb{R}}$. We define the upper and lower half-relaxed limits $u^*$ and $u_*$ of $u_\alpha$ as
\[
u^*(x)={\limsup}^*_{\alpha \rightarrow \infty} u_\alpha(x):=\lim_{\alpha\rightarrow\infty} \sup \{u_\beta(y):\|x-y\|_2\leq\frac{1}{\beta},\beta\geq \alpha\},
\]
and
\[
u_*(x)={{\liminf}_*}_{\alpha \rightarrow \infty} u_\alpha(x):=\lim_{\alpha\rightarrow\infty} \inf \{u_\beta(y):\|x-y\|_2\leq\frac{1}{\beta},\beta\geq \alpha\}.
\]
\end{definition}
Let us state one of the main results, the convergence of $Q^L$ to $Q$ as $L$ grows to infinity.
\begin{theorem}\label{thm:conv}
Let $Q^L:=Q^L(x,a):\mathbb{R}^n \times \mathbb{R}^m \rightarrow \mathbb{R}$ be a viscosity solution to
\begin{equation*}
\gamma Q^L - D_x Q^L \cdot f(x,a) - L\|D_a Q^L\|_2 -r(x,a)=0.
\end{equation*}
Then we have that
\[
Q^L(x,a) \rightarrow Q(x)\quad\text{locally uniformly}\quad\text{as}\quad L\rightarrow \infty,
\]
where $Q:=Q(x)$ is a unique viscosity solution to
\[
\gamma Q - \sup_{a\in\mathbb{R}^m} (D_x Q \cdot f(x,a)+r(x,a))=0.
\]
\end{theorem}
\begin{proof}
By the comparison principle, we have that
\[
-\frac{\|r\|_{L^\infty(\mathbb{R}^n\times \mathbb{R}^m)}}{\gamma} \leq Q^L(x,a) \leq \frac{\|r\|_{L^\infty(\mathbb{R}^n\times \mathbb{R}^m)}}{\gamma}.
\]
Noticing that $\|D_a Q(x)\|_2=0$, $Q(x)$ satisfies 
\[
\gamma Q(x) - D_x Q \cdot f(x,a) - r(x,a) -L\|D_a Q(x)\|_2 \geq 0,
\]
which implies that
\[
Q^L(x,a) \leq Q(x).
\]
Since $Q^L$ is bounded, we also get
\[
Q^* := {\limsup}^*_{L\rightarrow \infty } Q^L(x,a) \leq Q(x).
\]
Similarly,
\[
Q_* := {{\liminf}_*}_{L\rightarrow \infty } Q^L(x,a).
\]
Our goal is to show $Q^* = Q_* =Q(x)$, which is equivalent to
\[
\lim_{L\rightarrow \infty} Q^L(x,a)=Q(x).
\]
Rewriting~\eqref{eq:ql}, one can see that $Q^L$ is a viscosity solution to
\[
-\|D_a Q^L\|_2+\frac{1}{L}(\gamma Q^L - D_x Q^L \cdot f(x,a)-r(x,a))=0.
\]
Taking $L\rightarrow\infty$ and invoking the stability result on the viscosity solution~\cite{achdou2013introduction}, the lower half-relaxed limit $Q_*$ is indeed a supersolution to
\[
-\|D_a Q_*\|_2 \geq 0,
\]
and hence, $|D_a Q_*| \leq 0$. Therefore $Q_*$ is  independent of $a$, that is,
\[
Q_*(x,a) \equiv Q_*(x).
\]
Additionally, $Q^L$ also satisfies
\[
\gamma Q^L -D_x Q^L \cdot f(x,a) - r(x,a) = L \|D_a Q^L\|_2 \geq 0.
\]
Again by the stability property of viscosity solutions, taking the limit of $L$, we derive that $Q_*$ is a supersolution to
\[
\gamma Q_* - D_x Q_* \cdot f(x,a) -r(x,a) \geq 0,
\]
and 
\[
\gamma Q_* - \sup_{a\in\mathcal{R}^m}\{D_x Q_* \cdot f(x,a) -r(x,a) \}\geq 0,
\]
since $Q_*$ is independent of $a$. Therefore, again by the comparison principle,
\[
u_* \geq u,
\]
and we finish the proof as $Q_*=Q^*=Q$ implying
\[
\lim_{L\rightarrow \infty}Q^L(x,a) = Q(x).
\]
\end{proof}
The convergence property can be achieved only under the boundedness assumptions on the dynamics and reward functions, $f$ and $r$. One can show further how fast this convergence occurs introducing some structural assumptions on $f$ and $r$. 
\subsection{Rate of convergence for some special cases}
In the previous section, we show that $Q^L(x,a)$ converges to $Q(x)$ using the classical stability result on viscosity solutions. We now further investigate the rate of convergence of $Q^L(x,a) \rightarrow Q(x)$ as $L\rightarrow \infty$ under appropriate assumptions. 
We begin by showing the existence of optimal control for $Q$ and $Q^L$ based on~\cite{bardi1997optimal}. 
To prove the existence of an optimal control achieving the optimal value $Q$, we enforce the following assumption.
 \begin{assumption}\label{ass:C2_assumption}
A Borel 
function $\alpha(x,p):\mathbb{R}^{n}\times \mathbb{R}^{n} \to \mathbb{R}^{m}$ is uniquely determined to satisfy
$$
f(x,\alpha(x,p))\cdot p + r(x,\alpha(x,p))
 = \max_{a\in \mathbb{R}^{m}}(f(x,a)\cdot p + r(x,a)).
$$
We  further assume that there exists $C>0$, satisfying
\begin{align*}   
 \|f(x+h)-2f(x,a)+f(x-h,a)\|_2&\leq C\|h\|^2_{2},\\
  r(x+h)-2r(x,a)+r(x-h,a) &\geq -C\|h\|^2_{2}.
\end{align*}
 \end{assumption}

\begin{proposition}\label{prop:existence_Q}
Let $\gamma > 2\|f\|_{\text{Lip}(\mathbb{R}^n \times \mathbb{R}^m)}$ and Assumption \ref{ass:C2_assumption} be enforced. For any given  $x\in\mathbb{R}^n$, there exists $a(\cdot)\in\mathcal{A}$ from~\eqref{eq:controlset} such that
\begin{equation}\label{eq:L_optimal_reward}
Q(x)=\int_0^\infty e^{-\gamma s } r(x(s),a(s)) \mathrm{d}s,
\end{equation}
where
\begin{equation}\label{eq:L_optimal_dynamic}
\begin{cases}
x'(s) = f(x(s),a(s))\quad\text{for}\quad s>0,\\
x(0)=x.\\
\end{cases}
\end{equation}
\end{proposition}
\begin{proof}

By \cite{tran2021hamilton}[Theorem 2.8],  we have $Q(x)$ is bounded and Lipschitz continuous in $\mathbb{R}^n$. Denoting  $\tilde Q(x):=-Q(x)$, we have
\[
\tilde Q(x)+H_\gamma(x,D_x \tilde Q(x))=0,
\]
where
\[
H_\gamma (x,p):= \sup_{a\in \mathbb{R}^{m}} \left(-\frac{f(x,a)}{\gamma}\cdot p + \frac{r(x,a)}{\gamma} \right).
\]
We will show that $\tilde Q^L(x)$ is semiconcave.
By Assumption \ref{ass:C2_assumption}, for all $x, p,  q \in \mathbb{R}^{n}$,  we find
\begin{align*}
    H_\gamma(x,p)-H_\gamma(x,q) &  =  \sup_{a\in \mathbb{R}^{m}} \left(-\frac{f(x,a)}{\gamma}\cdot p + \frac{r(x,a)}{\gamma} \right) - \sup_{a\in \mathbb{R}^{m}} \left(-\frac{f(x,a)}{\gamma}\cdot q
    + \frac{r(x,\alpha(x,p) )}{\gamma} \right)\\
    &\leq   \left(-\frac{f(x,\alpha(x,p) )}{\gamma}\cdot p + \frac{r(x,\alpha(x,p) )}{\gamma} \right) - \left(-\frac{f(x,\alpha(x,p))}{\gamma}\cdot q + \frac{r(x, \alpha(x,p))}{\gamma} \right) \\
    &\leq  \left\|-\frac{f(x, \alpha(x,p))}{\gamma}\right\|_{2}\|p-q\|_{2}.
\end{align*}
We may assume that the left-hand side is positive. Otherwise, we change $p$ and $q$. Hence,
\[
|H_\gamma(x,p)-H_\gamma(x,q)| \leq C \|p-q\|_2 
\]
for some $C>0$. 
We claim that there exists $\tilde{C}>0$
\begin{equation}\label{eq:measure_H_gama_est}
H_\gamma(x+h,p+\tilde{C}h)-2H_\gamma(x,p)+H_\gamma(x-h,p-\tilde{C}h) \geq -\tilde{C} \|h\|_2^2,
\end{equation}
for all  $x,h\in \mathbb{R}^{n}$  and  $p\in \mathbb{R}^{n}$  satisfying $\|p \|_{2}\leq 3 \|Q(x,a)\|_{\text{Lip}(\mathbb{R}^{n} )}$. 
We observe that
\begin{align*}
&H_\gamma(x+h,p+\tilde{C}h)+H_\gamma(x-h,p-\tilde{C}h)\\
&
\geq  \sup_{ a\in \mathbb{R}^{m} } 
\left(-\frac{f(x+h,a)}{\gamma}\cdot (p+\tilde{C}h)
- \frac{f(x-h,a)}{\gamma}\cdot (p-\tilde{C}h)
+\frac{r(x+h,a)}{\gamma} 
+\frac{r(x-h,a)}{\gamma} 
\right).\\
&=\sup_{ a\in \mathbb{R}^{m} } 
\left(-\frac{f(x+h,a)+f(x-h,a)-2f(x,a)}{\gamma}\cdot p 
+\frac{r(x+h,a)+r(x-h,a)-2r(x,a)}{\gamma}\right.\\ 
&\qquad\qquad\qquad\qquad  \left. -\frac{f(x+h,a)-f(x-h,a)}{\gamma}\cdot  (\tilde{C}h )
-2\frac{f(x,a)}{\gamma }\cdot p
+2\frac{r(x,a)}{\gamma} \right)=:I.
\end{align*}
Note that we have 
\begin{align*}
\frac{\|f(x+h, a)-f(x-h, a)\|_{2}}{\gamma} &\leq   \frac{2\|f\|_{ \text{Lip} (\mathbb{R}^n\times \mathbb{R}^m) } \|h\|_{2}}{\gamma}.
\end{align*}
From the inequality above and the Assumption \ref{ass:C2_assumption},  we estimate $I$ by
\begin{align*}
I &\geq  - \left(C_1 ( \|Q(x)\|_{\text{Lip}(\mathbb{R}^{n})}+1 ) +-\frac{2\tilde{C}}{\gamma} \|f\|_{\text{Lip}(\mathbb{R}^{n}\times \mathbb{R}^{m} )} \right)\|h\|_{2}^2  +2 H_{\gamma}(x,p),
\end{align*}
for some constant $C_1>0$.
By choosing $\tilde{C}$ large enough to satisfy
\[
\tilde{C} > \frac{C_1 \gamma  (  \|Q(x)\|_{\text{Lip}(\mathbb{R}^{n})}+ 1 )    }{\gamma -2 \|f\|_{\text{Lip}(\mathbb{R}^{n}\times \mathbb{R}^{m} )} },
\]
the inequality \eqref{eq:measure_H_gama_est} holds. 
Note that such $\tilde{C}$ can be chosen since $\gamma>2\|f\|_{\text{Lip}
    (\mathbb{R}^n\times \mathbb{R}^m) }$.
Recalling~\cite{bardi1997optimal}[Theorem 4.9, Chapter II],  $\tilde{Q}(x)$ is semiconcave in $x$.

By \cite{bardi1997optimal}[Proposition 4.7, Chapter II], we have $D^{+}_{x} \tilde{Q} (x)=\partial_x \tilde{Q}(x) \supset D_x^* \tilde{Q}(x) \neq \emptyset$ for all $x\in \mathbb{R}^{n} $.\footnote{Here, \( D^{+}_xQ(x) \) and \( D^{-}_xQ(x) \)  denote the super-differential and sub-differential of the function \( Q(x) \).The notation \( \partial_x Q(x) \) refers to a generalized gradient or Clarke's gradient
and $D_x^* u = \{p\in \mathbb{R}^{n}:p=\lim D_xu(x_n), x_n\to x\}$.
For details about super-, sub-differential, and  Clarke's gradient, see \cite{bardi1997optimal}[Chapter II.4.1].}
Let us define a function $\psi:\mathbb{R}^{n} \to \mathbb{R}^{n}$ satisfying
\begin{equation}\label{eq:psi}
\psi(x) = \left\{
\begin{aligned}
    &D_x\tilde{Q} (x) \quad &\text{if $\tilde{Q}$ is differentiable at $x\in \mathbb{R}^{n}$},\\
    & \text{some } p \in D_x^+ \tilde{Q}(x) \cap D_x^* \tilde{Q}(x)  \quad &\text{otherwise}.
\end{aligned}
\right.
\end{equation}

For a given state $x(s)$, we choose a action 
$a(s) = \alpha(x(s), -\psi(x(s)))$.
For the moment, let us assume that we can choose $p\in D^+\tilde{Q}(x)\cap D_x^* \tilde{Q}(x) \cap D_x^* \tilde{Q}(x)$ so that $\psi(x)$ is a Borel function. Then clearly, we have  $a(\cdot) \in \mathcal{A}$. The details of the proof can be found in Appendix \ref{subsec:borel_psi}.   
Then from  Assumption \ref{ass:C2_assumption} and the definition of $\psi$, we have 
\begin{align*}
   &f(x(s), a(s))  \cdot \psi(x(s))    -r(x(s), a(s) ) \\
   & \quad =  -   \left( f(x(s), a(s))  \cdot (-\psi(x(s)) )   + r(x(s), a(s) ) \right)\\
   & \quad = - \max_{a\in \mathbb{R}^{m}} (f(x(s), a ) \cdot (-\psi(x(s)))+r(x(s), a))\\
   & \quad = \min_{a\in \mathbb{R}^{m}} (f(x(s), a ) \cdot \psi(x(s)) - r(x(s), a)).
\end{align*}
Since $\psi(x) \in D^+_x\tilde{Q}(x)$, the assumption of~\cite{bardi1997optimal}[Theorem 2.52, Chapter III] is satisfied.\footnote{Although the action spaces are non-compact set, with the same argument, the results of \cite{bardi1997optimal}[Theorem 2.52, Chapter III] still hold.
} 
Therefore, we have
$$
\tilde{Q}(x) = \int_{0}^{\infty }e^{-\gamma s }(-r(x(s),a(s)) \mathrm{d}s,
$$
which implies
$$
Q(x) = \int_{0}^{\infty }e^{-\gamma s }r(x(s),a(s)) \mathrm{d}s.
$$
This completes the proof.
\end{proof}
In the following proposition, we show the existence of optimal control for $Q^L$. 
Some of the arguments 
 in the proof overlap with the previous proposition, but we shall provide details for the convenience of the reader.
\begin{proposition}\label{prop:existence}
Assume further that $\gamma > 2\|f\|_{\text{Lip}(\mathbb{R}^n \times \mathbb{R}^m)}$ and let $(x,a)\in\mathbb{R}^n \times \mathbb{R}^m$ be given. Then, there exists $a_L(s) \in \mathcal{A}^L$ from~\eqref{eq:lipcontrol} such that
\begin{equation}\label{eq:L_optimal_reward}
Q^L(x,a)=\int_0^\infty e^{-\gamma s } r(x_L(s),a_L(s)) \mathrm{d}s,
\end{equation}
where
\begin{equation}\label{eq:L_optimal_dynamic}
\begin{cases}
x_L'(s) = f(x_L(s),a_L(s))\quad\text{for}\quad s>0,\\
x_L(0)=x,\\
a_L(0)=a.
\end{cases}
\end{equation}
\end{proposition}
\begin{proof}
We first prove the theorem under the assumption 
\begin{equation}\label{eq:f_r_c2}
    f(x,a), r(x,a) \in C^2(\mathbb{R}^{n}\times \mathbb{R}^{m})\quad \text{and} \quad 
    \|f\|_{C^2(\mathbb{R}^{n} \times \mathbb{R}^{m} )}+\|r\|_{C^2(\mathbb{R}^{n} \times \mathbb{R}^{m} )} \quad \text{is bounded}.
\end{equation}
Applying Lemma \ref{lem:Lip_regualrity}, we have $Q(x, a)$ is bounded and Lipscthiz continuous in $\mathbb{R}^n \times \mathbb{R}^m$. We introduce extended state variable $z(s)$ and dynamics $G:(\mathbb{R}^{n}\times \mathbb{R}^{m})\times \mathbb{R}^{m} \rightarrow \mathbb{R}^{n}\times \mathbb{R}^{m}$ satisfying 
$$
z(s) = (x_L(s), a_L(s) ) \quad \text{and} \quad  G(z,b) = (f(z), b).
$$
Then \eqref{eq:L_optimal_dynamic} can be written as 
\begin{equation*}
\begin{cases}
z'(s) &= G(z(s), b(s) )\quad\text{for}\quad s>0,\\
z(0)          &= (x,a).
\end{cases}
\end{equation*}
Rewriting \eqref{eq:ql}, in terms of $z(t)$ and $G(z,b)$ and $\tilde Q^L(x,a):=-Q^L(x,a)$, we have
\[
\tilde Q^L(x,a)+H_\gamma(z,D_z \tilde Q^L(z))=0,
\]
where
\[
H_\gamma (z,p):= \sup_{\|b\|_2\leq L} \left(-\frac{G(z,b)}{\gamma}\cdot p + \frac{r(z)}{\gamma} \right).
\]
We will show that $\tilde Q^L(z)$ is semiconcave.

It is clear that
\[
|H_\gamma(z,p)-H_\gamma(z,q)| \leq C \|p-q\|_2,
\]
for some $C>0$. 
We claim that there exists $\tilde{C}>0$ such that
\begin{equation}\label{eq:H_gama_est}
H_\gamma(z+h,p+\tilde{C}h)-2H_\gamma(z,p)+H_\gamma(z-h,p-\tilde{C}h) \geq -\tilde{C} \|h\|_2^2
\end{equation}
for all  $z,h\in \mathbb{R}^{n}\times\mathbb{R}^{m}$  and $p\in \mathbb{R}^{n} \times \mathbb{R}^{m}$  satisfying $\|p \|_{2} \leq 3 \|Q^L(x,a)\|_{\text{Lip}(\mathbb{R}^{n} \times \mathbb{R}^{m})}$. For justification, let us first observe that
\begin{align*}
&H_\gamma(z+h,p+\tilde{C}h)+H_\gamma(x-h,p-\tilde{C}h)\\
&
\geq  \sup_{\|b\|_2\leq L} 
\left(-\frac{G(z+h,b)}{\gamma}\cdot (p+\tilde{C}h)
- \frac{G(z-h,b)}{\gamma}\cdot (p-\tilde{C}h)
+\frac{r(z+h)}{\gamma} 
+\frac{r(z-h)}{\gamma} 
\right).\\
&=\sup_{\|b\|_2\leq L} 
\left(-\frac{G(z+h,b)+G(z-h,b)-2G(z,b)}{\gamma}\cdot p 
+\frac{r(z+h)+r(z-h)-2r(z)}{\gamma}\right.\\ 
&\qquad\qquad\qquad\qquad  \left. -\frac{G(z+h,b)-G(z-h,b)}{\gamma}\cdot  (\tilde{C}h )
-2\frac{G(z,b)}{\gamma }\cdot p
+2\frac{r(z)}{\gamma} \right)=:I.
\end{align*}
By the mean value theorem, we find 
\begin{align*}
       \frac{\|G(z+h, b) -2G(z,b)+G(z-h,b) \|_{2}}{\gamma}   &\leq   \frac{C\|f\|_{C^2(\mathbb{R}^{n}\times \mathbb{R}^{m})}\|h\|_{2}^2  }{\gamma },\\
     \frac{|r(z+h)-2r(z)+r(z-h)|}{\gamma} &\leq  \frac{C\|r\|_{C^2(\mathbb{R}^{n}\times \mathbb{R}^{m})}\|h\|_2^2  }{\gamma },\\
\frac{\|G(z+h)-G(z-h)\|_{2}}{\gamma} &\leq   \frac{2\|f\|_{ \text{Lip} (\mathbb{R}^n\times \mathbb{R}^m) } \|h\|_2}{\gamma}.
\end{align*}
From the inequalities above, we estimate $I$ by
\begin{align*}
I &\geq  -\left(\frac{C_1}{\gamma}(\|f\|_{C^2(\mathbb{R}^{n}\times \mathbb{R}^{m})}\|Q^L(x,a)\|_{\text{Lip}(\mathbb{R}^{n} \times \mathbb{R}^{m})}+\|r\|_{C^2(\mathbb{R}^{n}\times \mathbb{R}^{m})} )+\frac{2\tilde{C}}{\gamma} \|f\|_{\text{Lip}(\mathbb{R}^{n}\times \mathbb{R}^{m} )} \right)\|h\|_2^2 \\
&\qquad +2 H_\gamma(z,p).
\end{align*}
By choosing $\tilde{C}$ large enough to satisfy
\[
\tilde{C} > \frac{C_1(\|f\|_{C^2(\mathbb{R}^{n}\times \mathbb{R}^{m})}\|Q^L(x,a)\|_{\text{Lip}(\mathbb{R}^{n} \times \mathbb{R}^{m})}+\|r\|_{C^2(\mathbb{R}^{n}\times \mathbb{R}^{m})} ) }{\gamma -2 \|f\|_{\text{Lip}(\mathbb{R}^{n}\times \mathbb{R}^{m} )} },
\]
the inequality \eqref{eq:H_gama_est} holds. 
Note that such $\tilde{C}$ can be choosen since $\gamma>2\|f\|_{\text{Lip}
    (\mathbb{R}^n\times \mathbb{R}^m) }$.
Recalling~\cite{bardi1997optimal}[Theorem 4.9, Chapter II],  $\tilde{Q}^L$ is semiconcave on $\mathbb{R}^n$. By \cite{bardi1997optimal}[Proposition 4.7, Chapter II], we have $D^{+}_z \tilde{Q}^L (z)=\partial_z \tilde{Q}^L(z)$ for all $z\in \mathbb{R}^{n}\times \mathbb{R}^{m}$.
Finally, by~\cite{kim2021hamilton}[Theorem 1], there exists an optimal control $a_L(\cdot) \in \mathcal{A}^L$ and a corresponding state $x_L(s)$ that satisfy \eqref{eq:L_optimal_reward} and \eqref{eq:L_optimal_dynamic}.

Next, we remove the assumption \eqref{eq:f_r_c2}. For all $\varepsilon>0$, let $\tilde{\eta}^{\varepsilon}$ be a mollifier, as defined in \eqref{eq:mollifier}, but with the domain $\mathbb{R}^{n}$ replaced by $\mathbb{R}^{n} \times \mathbb{R}^{m}$. By defining mollified functions 
\[
f_i^\varepsilon(z):= f_i * \tilde{\eta}^{\varepsilon}(z) 
\quad 
\text{and}
\quad
r_i^\varepsilon(z) := r_i\ast \tilde{\eta}^{\varepsilon}(z),
\]
where the subscript $i$ denotes the $i$th component.

From \cite{evans2022partial}[Theorem 6, Appendix C.4], we have $f^\varepsilon(x,a), r^\varepsilon(x,a)\in C^2(\mathbb{R}^{n}\times\mathbb{R}^{m})$ and
\[
f^\varepsilon(z) \to f(z),\quad r^{\varepsilon}(z) \to r(z) 
\quad \text{locally uniformly}.
\]
For all $z_1, z_2\in \mathbb{R}^{n}\times \mathbb{R}^{m}$, we find
\begin{align*}
    \|f^{\varepsilon}(z_1)-f^{\varepsilon}(z_2)\|_{2}&\leq 
    \int_{\mathbb{R}^{n}\times \mathbb{R}^{m} }
    \|(f(z_1-z)-f(z_2-z)) \tilde{\eta}^{\varepsilon}(z)\|_{2}\, dz \\
    &\leq \|f\|_{
    \text{Lip}{(\mathbb{R}^n\times \mathbb{R}^m)} }
    |z_1-z_2| \int_{\{\|z\|_{2} \leq \varepsilon\} } \tilde{\eta}^{\varepsilon} (z)\, dz \\
    & \leq \|f\|_{
    \text{Lip}{(\mathbb{R}^n\times \mathbb{R}^m)} }
    |z_1-z_2|.
\end{align*}
Therefore,  $\|f^{\varepsilon}\|_{
    \text{Lip}{(\mathbb{R}^n\times \mathbb{R}^m)} }\leq \|f\|_{
    \text{Lip}{(\mathbb{R}^n\times \mathbb{R}^m)} }$ and   as a result, 
$\gamma>2\|f^{\varepsilon}\|_{\text{Lip} (\mathbb{R}^n\times \mathbb{R}^m) }$
    holds for all $\varepsilon>0$. 
As a direct consequence of the previous argument, there exist  $x^\varepsilon_L(s)$ and $a^\varepsilon_L(\cdot)\in \mathcal{A}_L$ satisfying  \eqref{eq:L_optimal_reward} and \eqref{eq:L_optimal_dynamic} replacing $f, r$ and $Q^{L}$  by $f^\varepsilon, r^\varepsilon$ and $Q^{L,\varepsilon}$, respectively.

For all $s\in [0, s_0]$ with  any fixed $s_0 >0$, we have the following estimates:
\begin{align*}
& \|x_L^{\varepsilon}(s)\| +
\left\|\frac{d x_L^{\varepsilon}(s)}{ds}\right\|_{2} + \|a_L^{\varepsilon}(s)\|_{2} +\left\|\frac{d a_L^{\varepsilon}(s)}{ds}\right\|_{2}\\
&\leq \|x\|_{2} 
+
\int_{0}^{s_0} 
\left\|\frac{dx_L^{\varepsilon}(s)}{ds}\right\|_{2} \mathrm{d}s + \left\|\frac{dx_L^{\varepsilon}(s)}{ds}\right\|_{2} + \|a\|_{2} 
+
\int_{0}^{s_0} 
\left\|\frac{d a_L^{\varepsilon}(s)}{ds}\right\|_{2} \mathrm{d}s + \left\|\frac{d a_L^{\varepsilon}(s)}{ds}\right\|_{2}\\
&\leq \|x\|_{2}
+
(s_0+1)\frac{\gamma}{2} + \|a\|_{2}+ (s_0+1)L.
\end{align*}

Given $n\in\mathbb{N}$, we can find a sequence $\{\varepsilon^n_k\}_{k=1}^\infty \searrow 0$ such that
\[
a^{\varepsilon_k^n}_{L} \rightarrow \hat{a}^n_{L} 
\quad \text{and} \quad 
x^{\varepsilon_k^n}_{L} \rightarrow \hat{x}^n_{L} \quad \text{locally uniformly}\quad \text{on}\quad [0,n]\quad \text{as}\quad k\rightarrow \infty,
\]
by Arzel\'{a}–Ascoli theorem. By taking a further subsequence, we may assume that
\[
a^{\varepsilon_k}_{L} \rightarrow \hat{a}_{L} 
\quad \text{and} \quad 
x^{\varepsilon_k}_{L} \rightarrow \hat{x}_{L} \quad \text{locally uniformly}\quad \text{on}\quad [0,\infty)\quad \text{as}\quad k\rightarrow \infty.
\]
Clearly, $\hat a_L$ is $L$-Lipschitz continuous on $[0,\infty)$. From 
\[
x_L^\varepsilon(t)=\int_0^t f(\hat x_L^\varepsilon(s), \hat a_L^\varepsilon(s))ds+x,
\]
and the local uniform convergence, we also have that
\[
\hat x_L (t) = \int_0^t f(\hat x_L(s),\hat a_L(s))\mathrm{d}s+x.
\]

Finally, we will show that 
\[
Q^L(x,a)=\int_0^\infty e^{-\gamma s } r(\hat x_L(s),\hat a_L(s)) \mathrm{d}s
\]
using the stability property of the viscosity solution. To this end, let us define  $\hat{Q}^{L} (x,a)$ as 
\[
\hat{Q}^{L} (x,a) := \int_0^\infty e^{-\gamma s }r(\hat{x}_{L}(s), \hat{a}_{L}(s)) \, \mathrm{d}s.
\]
and $g_L(x,a,t): (\mathbb{R}^{n}\times \mathbb{R}^{m}) \times [0,\infty) \rightarrow \mathbb{R}^{n}\times \mathbb{R}^{m}$ 
satisfy $g_L(x,a,t)  = (\tilde{x}_L(t), \tilde{a}_L(t))$, where 
$\tilde{x}(t)$ and $\tilde{a}(t)$ satisfies 
\begin{equation*}
\begin{cases}
{\tilde x_L'(s)}= f(\tilde x_L(s),\tilde a_L(s)),\\
\tilde x_L(0)=x,\\
\tilde a_L(0)=a,
\end{cases}
\end{equation*}
together with $\|\dot {\tilde {a}}(s)\|_2 \leq L$ for all $s>0$.

For any compact set $K\subset \mathbb{R}^{n} \times \mathbb{R}^{m}$ and $s_0>0$ to  be choose later, we define $U_L:=g_L(K, [0,s_0])$. 
Then for all $(x,a)\in K$, we find that
\begin{align*}
    |Q^{L,\varepsilon_k}(x,a) -  \hat{Q}^{L}(x,a) |
    &\leq \int_0^\infty e^{-\gamma s }|r^{\varepsilon_k} (x_{L}^{\varepsilon_k}(s),a_{L}^{\varepsilon_k}(s)) -r (\hat{x}_{L}(s),\hat{a}_{L}(s))| \, d\mathrm{s}\\
    &\leq \int_0^{s_0} e^{-\gamma s }|r^{\varepsilon_k} (x_{L}^{\varepsilon_k}(s),a_{L}^{\varepsilon_k}(s)) -r (\hat{x}_{L}(s),\hat{a}_{L}(s))| \, d\mathrm{s} 
    +  C e^{-\gamma s_0 } \\
    &\leq \int_0^{s_0} e^{-\gamma s }|r^{\varepsilon_k} (x_{L}^{\varepsilon_k}(s),a_{L}^{\varepsilon_k}(s)) -r^{\varepsilon_k}(\hat{x}_{L}(s),\hat{a}_{L}(s))| \, d\mathrm{s} \\
    &\quad   +
    \int_0^{s_0} e^{-\gamma s }|r^{\varepsilon_k} (\hat{x}_{L}(s),\hat{a}_{L}(s)) -r (\hat{x}_{L}(s),\hat{a}_{L}(s))| \, d\mathrm{s} 
    +  C e^{-\gamma s_0 }  \\
    &\leq C_1 |s_0|  \left(\|x^{\varepsilon_k} -\hat{x}\|_{L^{\infty}([0,s_0]) }+\|a^{\varepsilon_k} -\hat{a}\|_{L^{\infty}([0,s_0]) }\right)  \\
    &\quad + |s_0|\|r^{\varepsilon_k} -r\|_{L^{\infty}(U_L) } +  C_2 e^{-\gamma s_0 },
\end{align*}
where the last inequality comes from 
\[
\|f^{\varepsilon}\|_{
    \text{Lip}{(\mathbb{R}^n\times \mathbb{R}^m)} }\leq \|f\|_{
    \text{Lip}{(\mathbb{R}^n\times \mathbb{R}^m)} }
\quad \text{and} \quad
\|r^{\varepsilon}\|_{
    \text{Lip}{(\mathbb{R}^n\times \mathbb{R}^m)} }\leq \|r\|_{
    \text{Lip}{(\mathbb{R}^n\times \mathbb{R}^m)} }.
\]
For small $\delta>0$ given,  we choose $s_0>0$ such that 
\[
C_2e^{-\gamma s_0 } <\frac{\delta}{2}.
\]
Then $U$ is determined and we choose $\varepsilon_k>0$ sufficiently small to satisfy 
\begin{align*}
     C_1|s_0|  \left(\|x^{\varepsilon_k} -\hat{x}\|_{L^{\infty}([0,s_0]) }+\|a^{\varepsilon_k} -\hat{a}\|_{L^{\infty}([0,s_0]) }\right)   + |s_0|\|r^{\varepsilon_k} -r\|_{L^{\infty}(U_L) } < \frac{\delta}{2}.
\end{align*}
This implies that $Q^{L,\varepsilon_k}(x,a)$ converges to $\hat{Q}^{L}(x,a)$ uniformly on compact sets. 
Moreover, $Q^{L,\varepsilon}$ is the unique viscosity solution of 
$$
\gamma Q^{L,\varepsilon} - f^{\varepsilon} (x,a)\cdot D_xQ^{L, \varepsilon} -L \|D_a Q^{L, \varepsilon}\|_{2} -r^{\varepsilon}(x,a)= 0 .
$$
Then, by the stability of the viscosity solution \cite{bardi1997optimal}[Proposition 2.2, Chapter II], we find that \(\hat{Q}^{L}(x,a)\) is the viscosity solution of \eqref{eq:ql}. Consequently, \(\hat{Q}^L = Q^L\). Since \(\hat{a}(\cdot) \in \mathcal{A}^{L}\) and \(\hat{x}(\cdot)\), \(\hat{a}(\cdot)\) satisfy the dynamic equation \eqref{eq:L_optimal_dynamic}, it follows that \(\hat{a}^L(\cdot)\) is the desired optimal policy. This completes the proof.
\end{proof}
We now get back to our original interest, the rate of convergence of $Q^L$ to $Q$ as $L\rightarrow \infty$. To understand such a quantitative property, we introduce structural assumptions on the dynamics and reward function.    

\begin{assumption}\label{ass:approx2}
There exists $C>0$ such that
\[
\|f(x,a)\|_2+|r(x,a)| \leq \frac{C}{\|a\|_2^\sigma}
\]
for some $\sigma>0$.
\end{assumption}
The following assumption indicates that an optimal control from Proposition~\ref{prop:existence_Q} is not changing abruptly. 

\begin{assumption}\label{ass:control}
Given $x\in\mathbb{R}^n$, let $a(\cdot) \in \mathcal{A}$ be an optimal control obtained in Proposition~\ref{prop:existence_Q}, that is,
\[
Q(x) =  \int_0^\infty e^{-\gamma s} r(x(s),a(s)) \mathrm{d}s.
\]
For this control, we assume that there exists $\beta>0$,
\[
\|\tilde a(s)-a^\varepsilon(s)\|_{L^2(\mathbb{R})} \leq O(\varepsilon^\beta).
\]
where $a^\varepsilon (s)  = \tilde a(s) * \eta^\varepsilon$ and $\tilde a(s)$ is odd extension of $a(s)$ \footnote{The measurable function $a:[0,\infty)\rightarrow \mathbb{R}^m$ can be extended to $\tilde a:(-\infty,\infty)\rightarrow \mathbb{R}^m$ in a way that $a(-s)=-a(s)$ and mollified with $\eta_\varepsilon$ with $n=1$ in~\eqref{eq:mollifier}.}.
\begin{remark}
If $a(\cdot)$ is piecewise continous and $|a(\cdot)| \leq M$ for some $M>0$, then $\beta=1$. In general if $\tilde a \in W^{k,2}(\mathbb{R})$ for some $k \geq 1$, then we can choose $\beta=k$.
\end{remark}
\end{assumption}

We now state the result on a rate of convergence of $Q^L \rightarrow Q$. The idea of the proof is that we truncate the optimal control $a(s)$ and consider the mollification of this function to relate it with the feasible class $\mathcal{A}^L$.
\begin{theorem}
Let $\gamma \geq \|f\|_{\text{Lip}(\mathbb{R}^n\times \mathbb{R}^m)}+1$ and $(x,a)\in\mathbb{R}^n\times \mathbb{R}^m$ be given. Under Assumption~\ref{ass:control}, we have that
\[
|Q^L(x,a) - Q(x)| = O\bigg(\frac{\log L}{L^{\sigma/2}}+(\frac{\log L}{\sqrt{L}})^{\beta}+\frac{1}{\sqrt{L}}\bigg).
\]
\end{theorem}
\begin{proof}
Throughout the proof, we will use the mollifier $\eta_\varepsilon$ defined for $n=1$ in~\eqref{eq:mollifier}. Let $(x(s),a(s))$ and $(x_L(s),a_L(s))$ be optimal state trajectory and control pairs associated with $Q(x)$ and $Q^L(x,a)$ respectively. 

Given $R>0$ and $a(s)=(a_1(s),...,a_m(s))$, define
\[
a_{i}^R(s):= a_i(s) \mathbf{1}_{|a_i(s)|\leq R} + R \mathbf{1}_{\{a_i(s) \geq R\}} - R \mathbf{1}_{\{a_i(s) \leq - R\}}
\]
and let $a^R(s):=(a^R_{1}(s),...,a^R_{m}(s))$.

Since $Q(x) \geq Q^L(x,a)$, we only need to estimate $Q(x)-Q^L(x,a)$. Fixing $T>0$, we have that
\begin{equation}\label{eq:roc}
\begin{split}
Q(x)-Q^L(x,a) & \leq \int_0^\infty e^{-\gamma s }\bigg(r(x(s),a(s))- r(x_L(s),a_L(s))\bigg)\mathrm{d}s\\
&\leq \int_0^T e^{-\gamma s} \bigg(r(x(s),a(s))-r(x(s),a^R(s))\bigg) \mathrm{d}s \\&\quad\quad+ \int_{0}^T e^{-\gamma s} \bigg(r(x(s),a^R(s)-r(x_L(s),a_L(s))\bigg) \mathrm{d}s + C e^{-T}
\end{split}
\end{equation}
for some $C$ depending only on $\gamma$ and $C_1$ from Assumption~\ref{ass:naive}.
We notice that the first term of~\eqref{eq:roc} can be bounded as
\begin{equation*}
\begin{split}
\int_0^T e^{-\gamma s} \bigg(r(x(s),a(s))-r(x(s),a^R(s))\bigg) \mathrm{d}s &\leq C T \min\{{\frac{1}{R^{\sigma}},\|a(s)-a^R(s)\|_2}\}\\
& \leq \frac{CT}{R^\sigma}.
\end{split}
\end{equation*}

For the second part of~\eqref{eq:roc}, we assume for the moment that $R$ and $\varepsilon$ satisfy 
\[ 
|\frac{d a^{R,\varepsilon}(s)}{ds}| \leq \frac{e R}{\varepsilon} = L,\]
where $a_i^{R,\varepsilon}(s) := a_i^R(s) * \eta_\varepsilon$, and hence, $a^{R,\varepsilon}(s) \in \mathcal{A}^{L}$. 
Therefore, the second term in~\eqref{eq:roc} is estimated as
\begin{equation*}
\begin{split}
\int_{0}^T e^{-\gamma s} \bigg(r(x(s),a^R(s))-r(x_L(s),a_L(s))\bigg) \mathrm{d}s &= \int_0^T e^{-\gamma s} \bigg(r(x(s),a^R(s))-r(x_L(s),a_L(s))\bigg) \mathrm{d} s \\
&\leq \underbrace{\int_0^T e^{-\gamma s}\bigg(r(x(s),a^R(s))-r(x^{R,\varepsilon}(s),a^{R,\varepsilon}(s)\bigg)\mathrm{d}s}_{(a)},
\end{split}
\end{equation*}
where $x^{R,\varepsilon}(s)$ satisfies
\[
\begin{cases}
 (x^{R,\varepsilon})' (s) =f(x^{R,\varepsilon}(s),a^{R,\varepsilon}(s))\quad\text{for}s>0,\\
x^{R,\varepsilon}(0)=x,\\
a^{R,\varepsilon}(0)=a.
\end{cases}
\]
Now observing $(a)$,
\begin{equation}\label{eq:a}
\begin{split}
(a) &\leq \int_0^T e^{-\gamma s} |r(x(s),a^R(s))-r(x(s),a^{R,\varepsilon}(s)) | \mathrm{d}s+\int_0^T e^{-\gamma s} |r(x(s),a^{R,\varepsilon}(s))-r(x^{R,\varepsilon}(s),a^{R,\varepsilon}(s))| \mathrm{d}s\\
&\leq  C\|a^R(s) - a^{R,\varepsilon}(s)\|^2_{L^2([0,\infty))} +\int_0^T e^{-\gamma s} |r(x(s),a^{R,\varepsilon}(s))-r(x^{R,\varepsilon}(s),a^{R,\varepsilon}(s))| \mathrm{d}s\\
&\leq C\|a(s) - a^{\varepsilon}(s)\|^2_{L^2([0,\infty))} +\int_0^T e^{-\gamma s} |r(x(s),a^{R,\varepsilon}(s))-r(x^{R,\varepsilon}(s),a^{R,\varepsilon}(s))| \mathrm{d}s,
\end{split}
\end{equation}
where $a^\varepsilon(s):=(a_1(s)*\eta_\varepsilon,...,a_m(s)*\eta_\varepsilon)$.

Let $w(s):=\frac{1}{2}\|x(s)-x^{R,\varepsilon}(s)\|_2^2$ and $\tilde C = 2 \|f\|_{\text{Lip}{(\mathbb{R}^n\times \mathbb{R}^m)}}$.
\begin{equation*}
\begin{split}
w'(s) &= (x(s)-x^{R,\varepsilon}(s)) \cdot \bigg(f(x(s),a(s)-f(x^{R,\varepsilon}(s),a^{R,\varepsilon(s)} \bigg)\\
&= (x(s)-x^{R,\varepsilon}(s)) \cdot \bigg(f(x(s),a(s))-f(x(s),a^{R,\varepsilon}(s))+f(x(s),a^{R,\varepsilon}(s))-f(x^{R,\varepsilon}(s),a^{R,\varepsilon}(s))   \bigg)\\
&\leq \|x(s)-x^{R,\varepsilon}(s)\|_2 \times \\
&\quad\quad\quad\bigg( \|f(x(s),a(s))-f(x(s),a^R(s))\|_2 +\|f(x(s),a^R(s))-f(x(s),a^{R,\varepsilon}(s))\|_2+\frac{\tilde C}{2}\|x(s)-x^{R,\varepsilon}(s)\|_2\bigg)\\
&\leq \tilde C w +  \|x(s)-x^{R,\varepsilon}(s)\|_2\bigg(\|f(x(s),a(s))-f(x(s),a^R(s))\|_2 +\|f(x(s),a^R(s))-f(x(s),a^{R,\varepsilon}(s))\|_2\bigg)\\
& \leq \tilde C w + 2w+ \frac{1}{2}\bigg( \|f(x(s),a(s))-f(x(s),a^R(s))\|^2_2 +\|f(x(s),a^R(s))-f(x(s),a^{R,\varepsilon}(s))\|^2_2 \bigg)\\
&\leq (\tilde C+2) w + \frac{C}{R^{2\sigma}}+\frac{\tilde C}{4}\|a^R(s) - a^{R,\varepsilon}(s)\|_2^2.
\end{split}
\end{equation*}
By the Gronwall's inequality, for any $t\in[0,T]$, we have that
\begin{equation*}
\begin{split}
e^{-(\tilde C+2) t} \|x(t)-x^{R,\varepsilon}(t)\|_2^2 &\leq C\bigg(\frac{1}{R^{2\sigma}}+\int_0^t \|a^R(s)-a^{R,\varepsilon}(s)\|_2^2 \mathrm{d}s \bigg)\\
&\leq C\bigg(\frac{1}{R^{2\sigma}}+ \|a(s)-a^\varepsilon(s)\|^2_{L^2([0,\infty))}\bigg).
\end{split}
\end{equation*}
Noticing that $2\gamma \geq \tilde C + 2$, we have
\begin{equation*}
\begin{split}
\int_0^T e^{-\gamma s} |r(x(s),a^{R,\varepsilon}(s))-r(x^{R,\varepsilon}(s),a^{R,\varepsilon}(s))| \mathrm{d}s & \leq C\int_0^T e^{-\gamma s}\|x(s)-x^{R,\varepsilon}(s)\|_2\mathrm{d}s\\
&\leq C \bigg(\int_0^T 1\mathrm{d}s\bigg)^{1/2} \bigg(\int_0^T e^{-2\gamma s}\|x(s)-x^{R,\varepsilon}(s)\|_2^2 \mathrm{d}s \bigg)^{1/2} \\
& \leq C T \bigg( \frac{1}{R^{2\sigma}}+\|a(s)-a^\varepsilon(s)\|^2_{L^2([0,\infty))}\bigg).
\end{split}
\end{equation*}

Finally, we set $R=\sqrt{L}$, $T=\log R$, and $\varepsilon=\frac{eR}{L}$,
\begin{equation*}
\begin{split}
Q(x)-Q^L(x,a) &\leq C(\frac{T}{R^\sigma}+T\|a(s)-a^\varepsilon(s)\|^2_{L^2([0,\infty))}+e^{-T})\\
&\leq O( \frac{T}{R^\sigma} + T\varepsilon^{2\beta} + \frac{1}{\sqrt{L}})\\
&\leq O\bigg(\frac{\log L}{L^{\sigma/2}}+(\frac{\log L}{\sqrt{L}})^{\beta}+\frac{1}{\sqrt{L}}\bigg).
\end{split}
\end{equation*}
\end{proof}


\section{Generalized Hamilton--Jacobi based Q-learning}\label{sec:gen}
The core idea of Hamilton--Jacobi based Q-learning for the Lipschitz continuous control problem proposed by~\cite{kim2021hamilton} is to restrict the derivative of the action $a(\cdot)$ to be bounded, that is, $|\dot a (\cdot)| \leq L$ for some $L>0$, which results in 
\[
\gamma Q^L - D_x Q^L \cdot f(x,a) - L\|D_a Q^L\|_2 -r(x,a)=0
\]
for $Q^L$ defined as~\eqref{eq:Q_L_dynamic}. We now generalize the constraint by introducing
\[
\mathcal{A}_p^{L}:=\{ a(\cdot) \in \mathcal{A}^L: \|\dot a(\cdot) \|_p \leq L \quad \text{for all}\quad s\in [0, \infty) \}
\]
for $p \in [1,\infty]$. Here, we note that $\mathcal{A}_2^L = \mathcal{A}^L$ when $p=2$.

The value function is defined as
\begin{equation}\label{eq:Q_L_dynamic_p}
Q_p^L(x,a) = \sup_{a\in \mathcal{A}_p^{L}}\left\{ \int_0^\infty e^{-\gamma s } r(x(s),a(s)) \mathrm{d}t : x(0)=x , a(0) =a\right\},
\end{equation}
and we obtain a slightly different Hamilton--Jacobi equation
\begin{equation*}
\begin{split}
\gamma Q^L - D_x Q^L \cdot f(x,a) -r(x,a) -\sup_{\|b\|_p \leq L} b \cdot D_a Q &=\gamma Q^L - D_x Q^L \cdot f(x,a) -r(x,a) -L \|D_a Q\|_q=0,
\end{split}
\end{equation*}
where $\frac{1}{p}+\frac{1}{q}=1$.

\subsection{Demonstration of optimal action}
Given $L$ and $p$, let us examine the structural property of optimal action when the value function $Q_p^L(x,a)$ is differentiable everywhere. 

Let us illustrate some special cases $p=1$ or $p=\infty$ first and consider general $p\in(1,\infty)$. When $p=1$, $Q^L:=Q^L_1$ is a viscosity solution to
\[
\gamma Q^L - D_x Q^L \cdot f(x,a) -r(x,a) - L \max_{i \in [1,m]} \{|D_{a_i} Q^L |\}=0.
\]
The optimal action $a(\cdot)=(a_1(\cdot),...,a_m(\cdot))$ satisfies that
\begin{equation*}
\dot a_j(\cdot)=\begin{cases} \pm L \quad&\text{if}\quad j=\max_{i\in[1,m]}|D_{a_i}Q^L(x(\cdot),a(\cdot))|,\\
 0 \quad&\text{if}\quad D_{a_i}Q^L(x(\cdot),a(\cdot))\neq 0 \quad\text{and}\quad i\neq j.
\end{cases} 
\end{equation*}

On the other hand, if $p=\infty$, we also derive that $Q^L := Q^L_\infty(x,a)$ solves
\[
\gamma Q^L - D_x Q^L \cdot f(x,a) -r(x,a) - L\sum_{i=1}^m |D_{a_i} Q^L |=0
\]
in the viscosity sense. Similarly, the optimal action satisfies that
\begin{equation*}
\dot a_i(\cdot)=\begin{cases} L \quad&\text{if}\quad D_{a_i}Q^L(x(\cdot),a(\cdot))>0,\\
 -L \quad&\text{if}\quad D_{a_i}Q^L(x(\cdot),a(\cdot))<0.
\end{cases} 
\end{equation*}
One interesting feature of implementing $\ell_\infty$ constraint is that the rate of change of control $\dot a_i (\cdot)$ is independent of the coordinate while that under $\ell_2$ constraint is not. Specifically, $\dot a_1(t)=L$ for some $t$ fixed does not affect the choice of $\dot a_2(t)$ under $\ell_\infty$ constraint but the condition forces $\dot a_i(t)=0$ for $i \neq 0$. 

We now characterize the optimal control associated with $p\in (1,\infty)$.
\begin{proposition}\label{prop:gen}
Let $p\in (1,\infty)$ and $Q^L_p(x,a)$ from~\eqref{eq:Q_L_dynamic_p} be differentiable. If $D_{a_i} Q^L_p(x,a) \neq 0$ for all $i$, then optimal action $a(\cdot)$ satisfies 
\[
\dot a_i(\cdot) = \pm L \frac{(D_{a_i} Q^L)^{1/(p-1)}}{\|D_a Q^L\|_q^{q/p}},
\]
where $\frac{1}{p}+\frac{1}{q}=1$.
\end{proposition}

\subsection{Convergence result}
Regardless of the choice of $p$, invoking the same argument presented in Theorem~\ref{thm:conv}, we still have the general convergence result. 
\begin{cor}
Let $p \in [1,\infty]$ and  $Q^L:=Q_p^L(x,a):\mathbb{R}^n \times \mathbb{R}^m \rightarrow \mathbb{R}$ be a viscosity solution to
\begin{equation}\label{eq:ql_q_norm}
\gamma Q^L - D_x Q^L \cdot f(x,a) - L\|D_a Q^L\|_q -r(x,a)=0
\end{equation}
for $q$ satisfying $\frac{1}{p}+\frac{1}{q}=1$.
Then we have that
\[
Q^L(x,a) \rightarrow Q(x)\quad\text{locally uniformly}\quad\text{as}\quad L\rightarrow \infty,
\]
where $Q:=Q(x)$ is a unique viscosity solution to
\[
\gamma Q - \sup_{a\in\mathbb{R}^m} (D_x Q \cdot f(x,a)+r(x,a))=0.
\]
\end{cor}

\section{Numerical experiment}\label{sec:num}
This section presents numerical experiments conducted using the HJDQN learning framework, with different values of $L$ and $p$.  We assess performance in the Hopper-v2, HalfCheetah-v2, and Walker2d-v2 environments from the OpenAI Gym suite \cite{brockman2016openai}, simulated using the MuJoCo engine \cite{todorov2012mujoco}. Additionally, we include a 20D linear quadratic regulator (LQR) control problem as a benchmark task. Our experiments are based on the code provided by \cite{kim2021hamilton}\footnote{\url{https://github.com/HJDQN/HJQ}}.
Unless otherwise stated, the hyperparameters remain consistent with those reported in \cite{kim2021hamilton}.

\subsection{Proposed algorithm}
Yet almost identical to the procedure presented in~\cite{kim2021hamilton}, we include the modified algorithm for completeness.

\begin{algorithm}
\caption{Modified Hamilton--Jacobi DQN ($p$-HJDQN)}
\label{alg:ula}
\begin{algorithmic}[1]
\State $L$, $p$ and $h$ are given;
\State {\bf Initialization:} Initialize $Q$-function by a neural network with random weight $\theta$ as $Q_\theta$, and the target $Q$-function as $Q_{\theta^-}$ with weights $\theta^{-}=\theta$;
\State Initialize replay buffer with fixed capacity;
\For{episode=1 to $M$}
\State Sample initial state-action pair $(x_0,a_0)$;
\For{k=0 to $K$}
    \State Execute action $a_k$ and observe reward $r_k$ and the next state $x_{k+1}$;
    \State Store $(x_k,a_k,r_k,x_{k+1})$ in buffer;
    \State Sample the random mini-batch \{($x_j,a_j,r_j,x_{j+1}$)\} from buffer;
    \State Compute $\eta$ via Proposition~\ref{prop:gen};
    \State Set $y_{j}^- := hr_j+(1-\gamma h )Q_{\theta^-}(x_{j+1},a_j')$ for all $j$ where $a_j':= a_j  \pm hL \frac{(D_{a_i} Q^L)^{1/(p-1)}}{\|D_a Q^L\|_q^{q/p}}$. 
    \State Update $\theta$ by minimizing $\sum_{j}(y_j^- - Q_\theta(x_j,a_j))^2$;
    \State $\theta^- \leftarrow (1-\alpha)\theta^- +\alpha \theta$ for $\alpha$ small positive;
    \State Set the next action as $a_{k+1}:=a_k+h \eta +\varepsilon$, where $\varepsilon\sim N(0,\sigma^2 I_m)$;
    \EndFor
\EndFor
\end{algorithmic}
\end{algorithm}

\subsection{Numerical analysis on convergence result}
This subsection provides empirical evidence on the convergence result in Lemma \ref{lem:stepwise_converge}. To this end, we respectively sample $500$ states and actions and  then compute $\max_{1\leq i,j\leq 500} |Q_{L}(x_i,a_j) -Q_{L+10}(x_i,a_j)|$ for different values of $L$ from $L=10,...,150$.

The numerical values of $\max_{1\leq i,j\leq 500} |Q_{L}(x_i,a_j) - Q_{L+10}(x_i,a_j)|$ for four different tasks are presented in Table \ref{table:Q_diff}. Figure \ref{fig:Q_dif_plot} illustrates the corresponding line graphs derived from the data in Table \ref{table:Q_diff}. A decreasing trend of $\|Q_{L+10}-Q_{L}\|_{\text{Lip}(\mathbb{R}^{n}\times \mathbb{R}^{m})}$ is observed as $L$ increases, providing empirical support for Lemma \ref{lem:stepwise_converge}. However, this convergence trend exhibits fluctuations and instability, particularly in the HalfCheetah-v2 and 20-dimensional LQ problem environments. Further investigation is warranted to understand the causes of instability and to refine the HJDQN learning framework.

\begin{figure}[h]
  \centering
  \includegraphics[width=\textwidth]{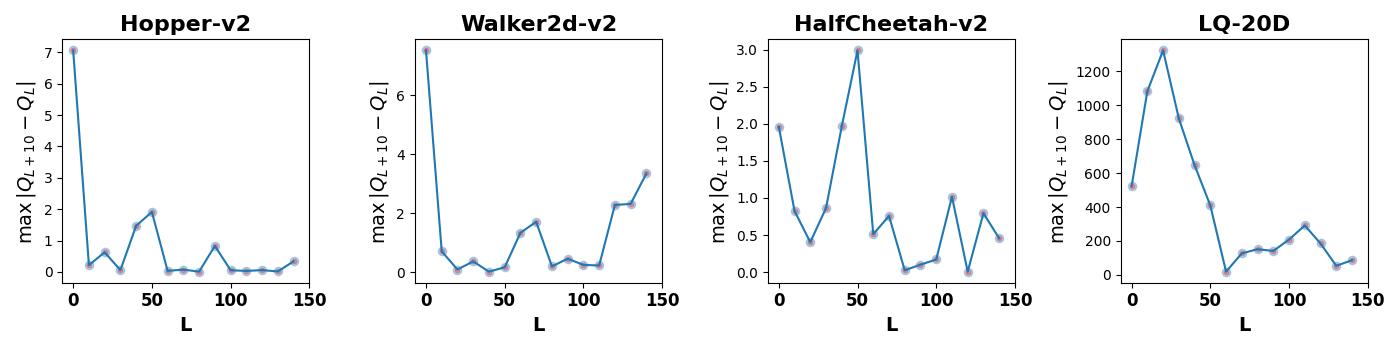}
  \caption{Plot for the Table \ref{table:Q_diff}}
  \label{fig:Q_dif_plot}
\end{figure}

\begin{table}[h]
\centering
\begin{tabular}{ccccc}
\hline
                    & Hopper-v2          & Walker2d-v2                 & HalfCheetah-v2  &  LQR-20D \\ \hline
$|Q_{20}-Q_{10}|$    & $7.0607$         & $7.5138$                 & $1.9608$  & $522.4116$   \\ 
$|Q_{30}-Q_{20}|$    & $0.2177$         & $0.7117$    & $0.8246$  & $1082.9729$     \\ 
$|Q_{40}-Q_{30}|$    & $0.6256$         & $0.0996$    & $0.4034$  & $1323.8116$    \\ 
$|Q_{50}-Q_{40}|$    & $0.0577$         & $0.3741$    & $0.8640$  & $922.2965$     \\ 
$|Q_{60}-Q_{50}|$    & $1.4704$         & $0.0268$    & $1.9657$   & $646.7611$    \\ 
$|Q_{70}-Q_{60}|$    & $1.9120$         & $0.1756$    & $2.9905$  & $408.9908$  \\ 
$|Q_{80}-Q_{70}|$    & $0.0351$         & $1.3410$    & $0.5117$   & $18.6145$   \\ 
$|Q_{90}-Q_{80}|$    & $0.0746$         & $1.7137$    & $0.7575$ & $126.5660$   \\ 
$|Q_{100}-Q_{90}|$   & $0.0093$         & $0.2049$    & $0.0248$  & $151.0665$   \\ 
$|Q_{110}-Q_{100}|$  & $0.8314$         & $0.4649$    & $0.1005$  & $140.7522$  \\ 
$|Q_{120}-Q_{110}|$  & $0.0557$         & $0.2555$    & $0.1706$  & $208.1184$       \\ 
$|Q_{130}-Q_{120}|$  & $0.0317$         & $0.2358$    & $1.0160$  & $291.1623$       \\ 
$|Q_{140}-Q_{130}|$  & $0.0586$         & $2.2847$    & $0.0063$  & $184.8106$       \\ 
$|Q_{150}-Q_{140}|$  & $0.0155$        & $2.3188$     & $0.7942$  & $53.1037$       \\ 
$|Q_{160}-Q_{150}|$  & $0.3373$        & $3.3487$     & $0.4554$  & $86.3219$. 
\end{tabular}
\caption{
Numerical calculation of $\|Q^{L+10} -Q^{L}\|_{L^\infty(\mathbb{R}^{n}\times \mathbb{R}^{m})}$ for different tasks.
}
\label{table:Q_diff}
\end{table}

\subsection{Numerical analysis on different values of $1\leq p\leq \infty$}

In this section, we explore the impact of varying $p$ values on the results. We examine $p$ values of $1, 2, 10, 100$ and $\infty$. As a benchmark, we utilize the DDPG algorithm, as introduced in \cite{lillicrap2015continuous}. We present the learning curves for different tasks and $p$ values below.

It's noteworthy that in the HalfCheetah-v2 environment, DDPG consistently outperforms all HJDQN variants for $p = 1, 2, 10, 100$ and $\infty$. Also, the standard HJDQN setting ($p=2$) demonstrates superior performance over other $p$ values.

Conversely, the HJDQN framework exhibits enhanced performance in the Walker2d-v2 and 20-dimensional LQ problems. Specifically, in Walker2d-v2, we observe an upward trend in the average return as $p$ increases, highlighting the influence of $p$ can differ across various tasks.

Subsequently, we compare action trajectories between DDPG and HJDQN across different $p$ values, as illustrated in Figure \ref{fig:cheetah_action}, \ref{fig:walker_action} and \ref{fig:lqr-1}, \ref{fig:lqr-2} and \ref{fig:lqr-3}. A notable observation is the increased oscillation frequency in the actions with rising $p$ values, indicating a potential impact on system stability or efficiency.


\begin{figure}[h!]
  \centering  \includegraphics[width=\textwidth]{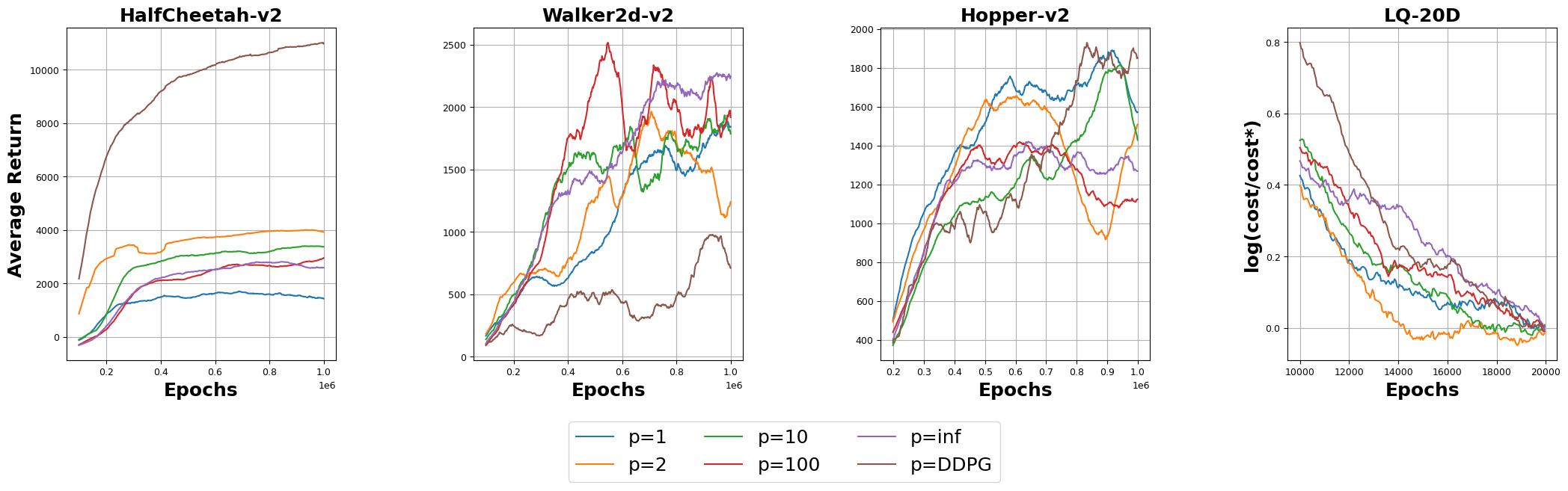}
  \caption{Learning curves for different tasks.}
  \label{fig:Q_dif_plot}
\end{figure}

\begin{figure}[h!]
  \centering  \includegraphics[width=\textwidth]{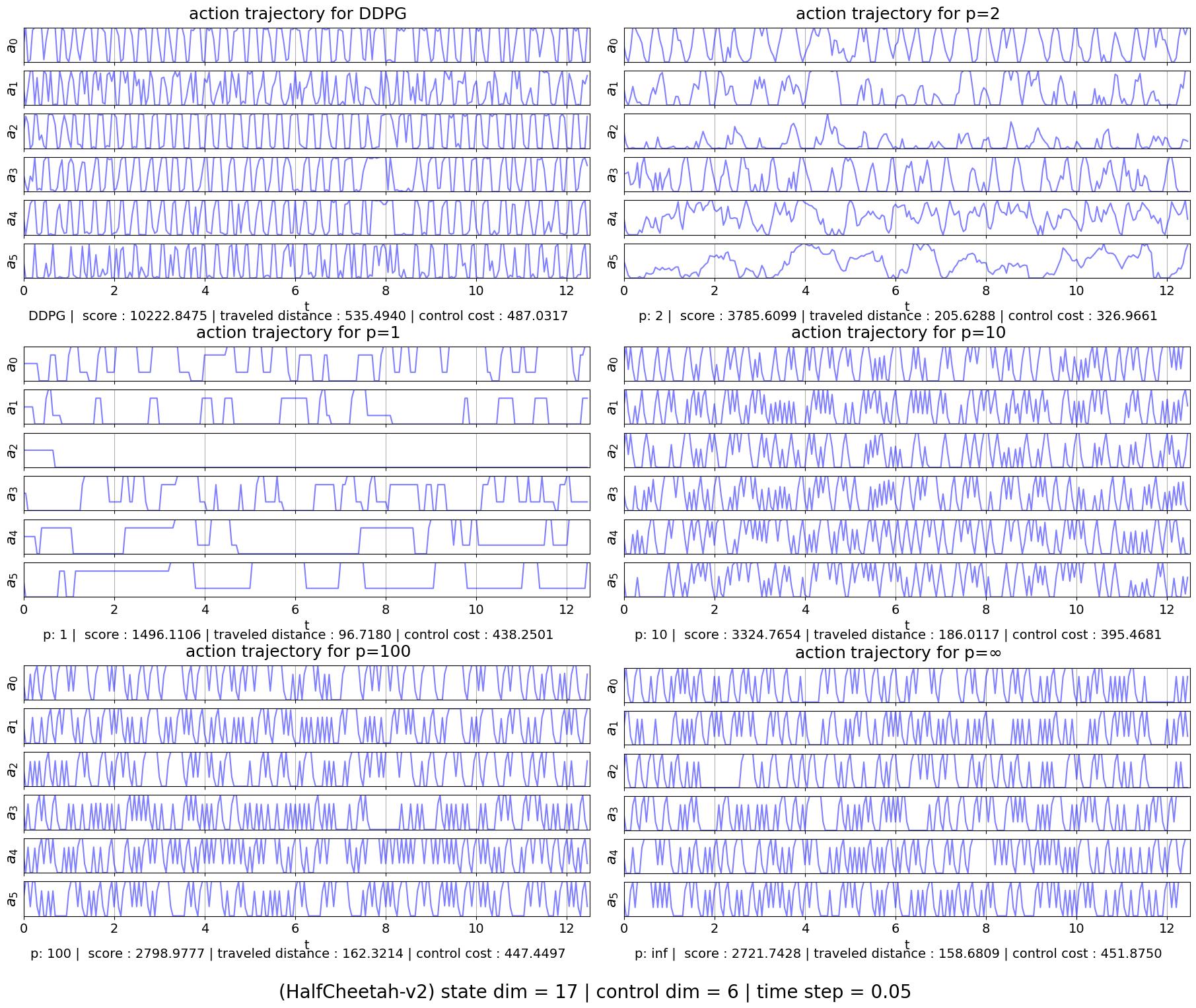}
  \caption{Action trajectories obtained by different values of $p$ in  HJDQN and DDPG for HalfCheetah-v2.}
    \label{fig:cheetah_action}
\end{figure}


\section{Discussions}\label{sec:con}
We study the behavior of the value function for optimal control problems where the control is restricted to be $L$-Lipschitz continuous, which arises in a novel HJB-based $Q$-learning. By augmenting the control variable, we formulate the problem in the standard optimal control problem where the control takes value in a compact set and investigates the effect of $L$, particularly, stability properties and the rate of convergence. We extend the theory by introducing $p$-norm and discover that different choice of $p$ indeed leads to different features of optimal control. We observe that the larger $p$ induces a more frequent change of control. 

The choice of $p$ is subtle. HJDQN with $p=2$ and DDPG are not always leading to.

{\bf Conflict of interest}
The authors declare that they have no conflict of interest.

{\bf Data availability}
 Data sharing does not apply to this article as no datasets were
generated or analyzed during the current study.


\bibliographystyle{amsplain}
\bibliography{references.bib}

\section{Appendix}\label{sec:appendix}
\subsection{Existence of $\psi$ in \eqref{eq:psi}}\label{subsec:borel_psi}
In this subsection, we prove the existence a Borel function satisfying \eqref{eq:psi} using  Borel selection theorem for set-valued functions. 
Since  $\tilde{Q}$  is a semiconcave, there exists a constant $C>0$ and a convex function $u:\mathbb{R}^{n} \to \mathbb{R}$ such that 
$$
\tilde{Q}(x) = C|x|^2 -u(x).
$$
It is enough to show that there exists a Borel function satisfying $g(x) \in D_x^* u(x)$.
We denote by $2^{\mathbb{R}^{n}}$ the power set of $\mathbb{R}^{n}$.
Let us define  a set-valued function $T: \mathbb{R}^{n} \to 2^{\mathbb{R}^{n}}$ by $T(x) =D^-_xu(x)$ and $T_{D^*_xu}(x) = D^*_xu(x)$.

By \cite{borwein2010convex}[Proposition 6.1.1],
for all closed set $F\subset \mathbb{R}^n$, we have  
$$
T^{-1}(F):= \{x\in \mathbb{R}^{n} : T(x)\cap F\not =\emptyset 
\}\quad \text{is closed.}
$$
From the definition of $D^*_xu$, it is clear that $T_{D_x^*}(x)$ is a closed multi-valued function, see \cite{border2013introduction}[Definition 16], for the details.
Since $\tilde{Q}$ is Lipschitz for each $x\in \mathbb{R}^{n}$, so is $u$ which implies that  $T(x)$ is compact-valued for each $x\in \mathbb{R}^{n}$.
By the definition, for $\tilde T(x):=T(x)\cap T_{D_x^*}(x)$, $\tilde T(x)=T_{D_x^*}(x)$ and it also satisfies the condition
$$
\tilde T^{-1}(F)= \{x\in \mathbb{R}^{n} : \tilde{T}(x)\cap F\not =\emptyset 
\}\quad \text{is closed for any closed set $F$},
$$
by \cite{guide2006infinite}[Theorem 17.25].
Next, according to the Kuratowski--Ryll-Nardzewski Selection Theorem, as referenced in \cite{guide2006infinite}[18.13], there exists a Borel function such that \( g(x) \in D_x^{*} u \), which completes the proof. For more details, we refer to \cite{guide2006infinite}[Chapters 17 and 18],  \cite{borwein2010convex}[Chapter 6] and \cite{border2013introduction}, which discuss selection theory related to set-valued functions and convex functions.

\subsection{Action plots on Walker2d-v2 and 20-dimensional LQ problems}
We provide action trajectories plots for Walker-2D and 20-dimensional LQ problems in this subsection. 

\begin{figure}[h!]
  \centering  \includegraphics[width=\textwidth]{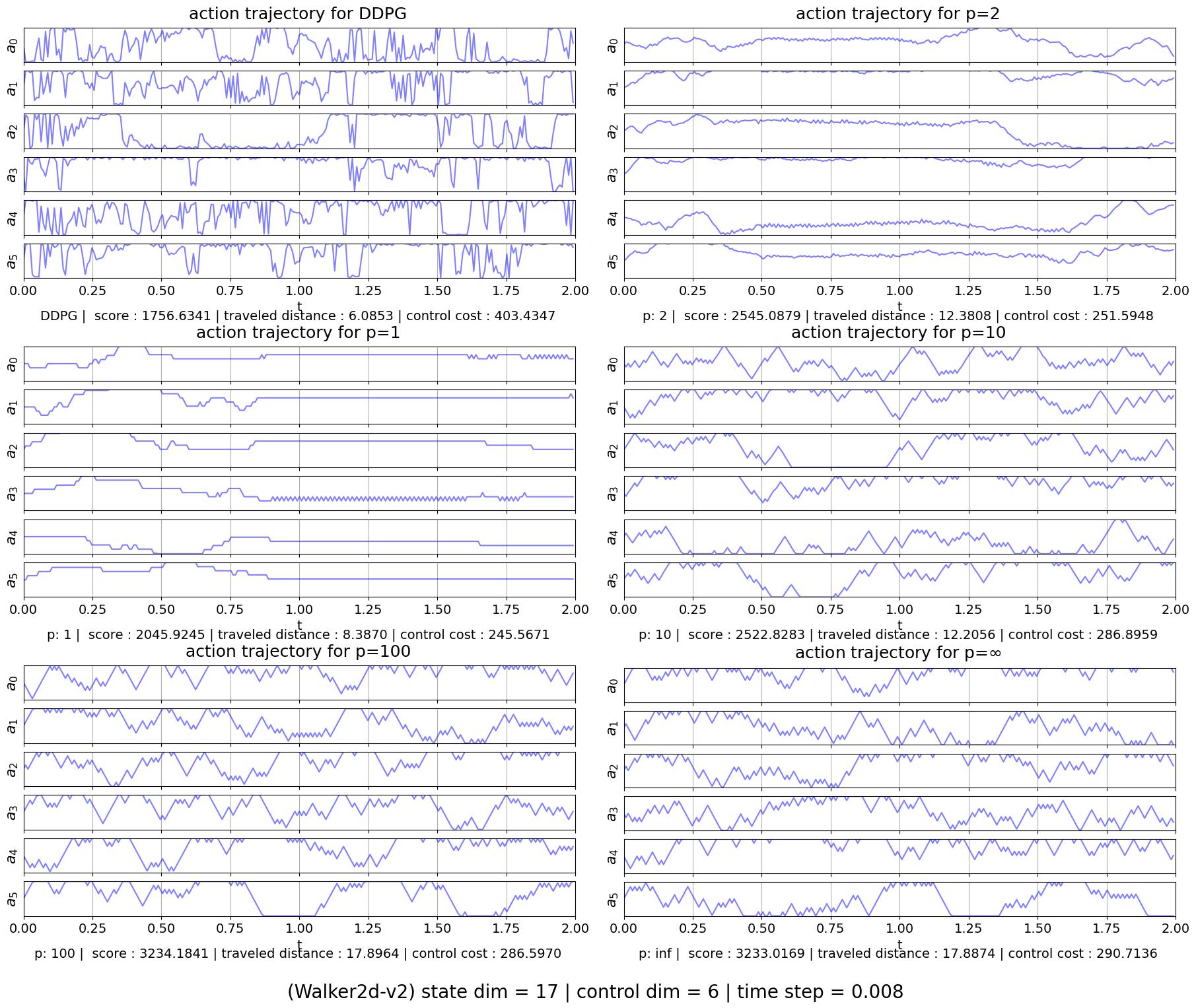}
  \caption{ Action trajectories obtained by different values of $p$ in  HJDQN and DDPG for Walker2d-v2.}
      \label{fig:walker_action}
\end{figure}

\begin{figure}[h!]
  \centering  \includegraphics[width=\textwidth]{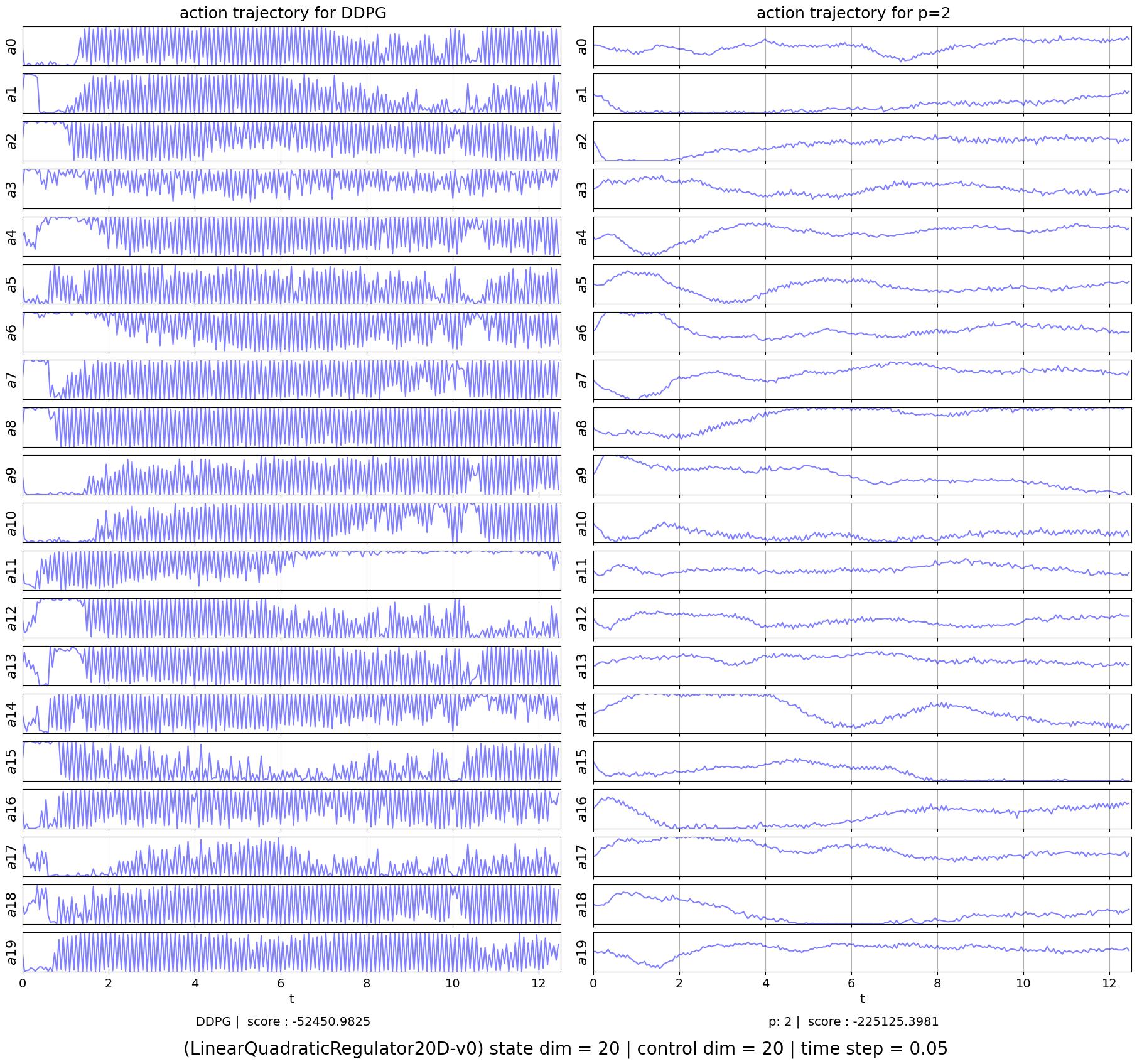}
  \caption{Action trajectories obtained by  HJDQN with $p=2$ and DDPG for 20D LQR problem.}
      \label{fig:lqr-1}
\end{figure}

\begin{figure}[h!]
  \centering  \includegraphics[width=\textwidth]{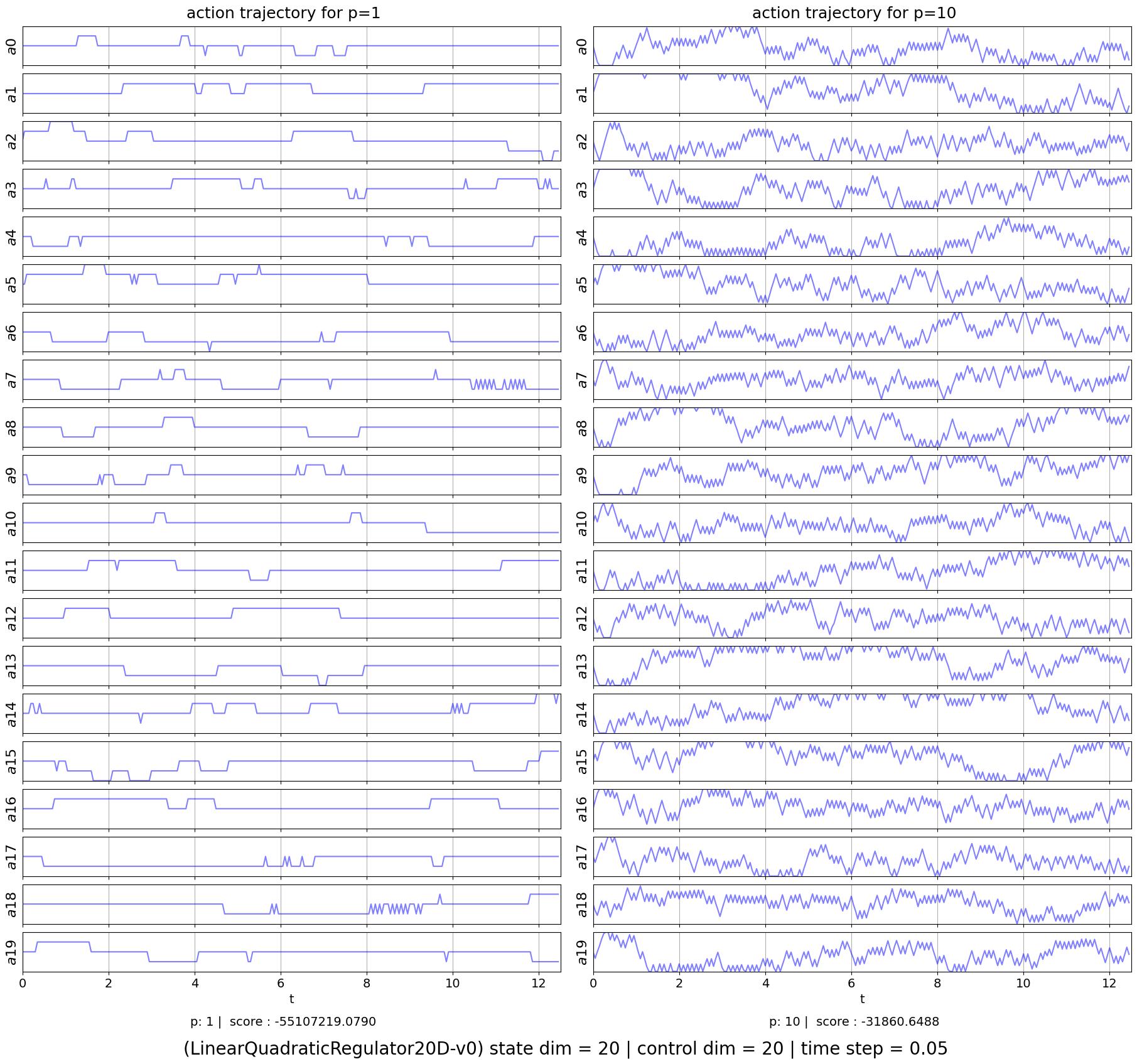}
  \caption{Action trajectories obtained by  HJDQN with $p=1$ and $p=10$ for 20D LQR problem.}
      \label{fig:lqr-2}
\end{figure}

\begin{figure}[h!]
  \centering  \includegraphics[width=\textwidth]{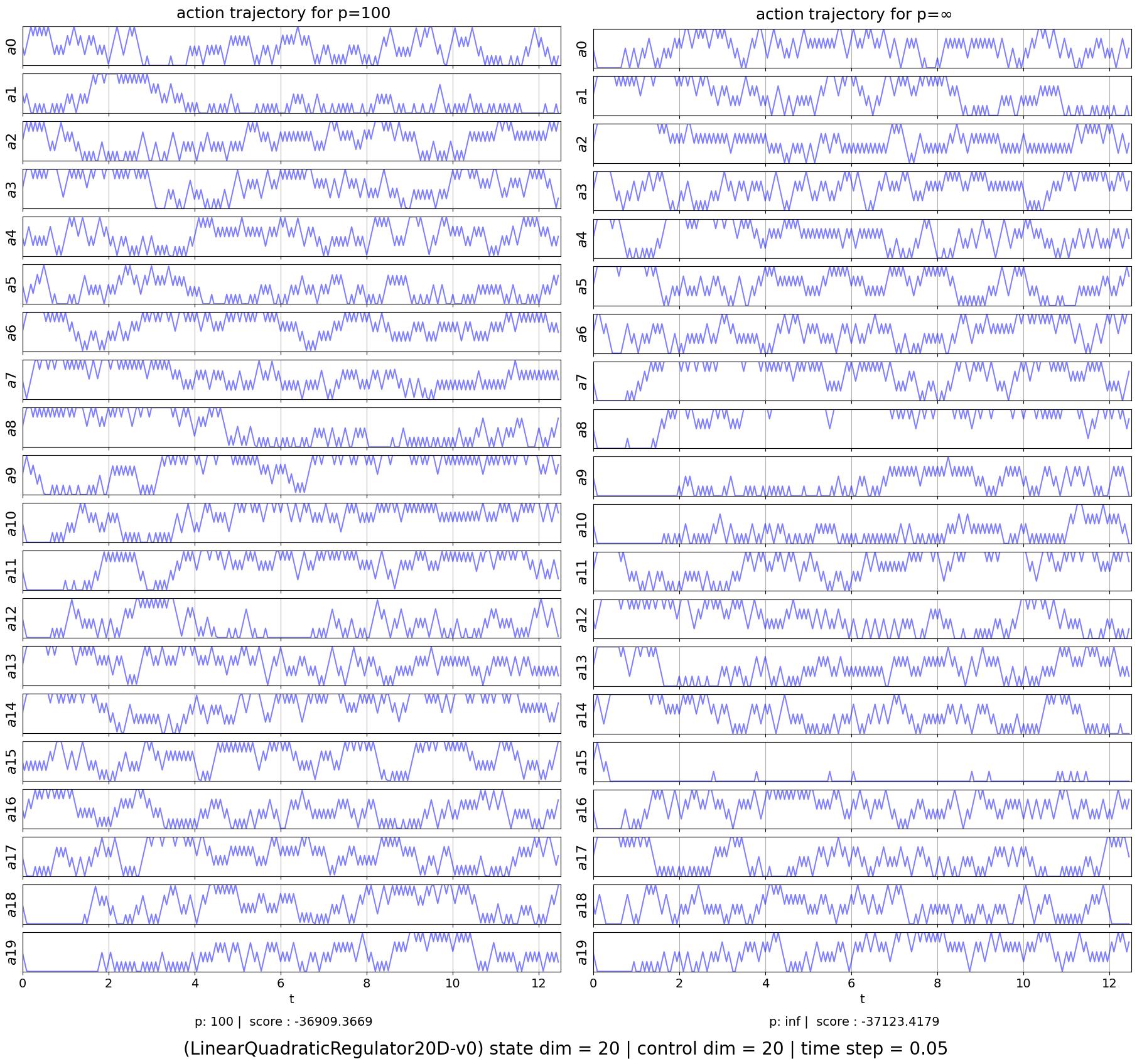}
  \caption{Action trajectories obtained by  HJDQN with $p=100$ and $\infty$ for LQR problem with 20 dimension.}
      \label{fig:lqr-3}
\end{figure}

\end{document}